\documentclass[12pt,a4paper,twoside]{article}
\usepackage{a4wide,amsmath,amsbsy,amsfonts,amssymb,stmaryrd,amsthm,mathrsfs,graphicx}
\newcommand\ZZ{{\hat{\mathbb Z}}}
\newcommand\Z{{\mathbb Z}}

\newcommand\Q{{\mathbb Q}}
\newcommand\F{{\mathbb F}}

\newcommand\QQ{\overline{\mathbb Q}}
\newcommand\K{\Bbbk}

\newcommand\C{{\mathbb C}}
\newcommand\N{{\mathbb N}}
\newcommand\cL{{\mathscr L}}
\newcommand\Tau{{\mathscr T}}
\newcommand\hL{\hat{\mathscr L}}
\newcommand\cP{{\mathscr P}}

\newcommand\ra{\rightarrow}

\newcommand\ilim{\lim\limits_{\longleftarrow}\,}

\newcommand\Sp{\mathrm{Sp}}
\newcommand\Spec{\mathrm{Spec}}
\newcommand\aut{\mathrm{Aut}}
\newcommand\out{\mathrm{Out}}
\newcommand\inn{\mathrm{inn}}

\newcommand\hookra{\hookrightarrow}
\newcommand\tura{\twoheadrightarrow}
\newcommand\da{\downarrow}

\newcommand\dd{\partial}
\renewcommand{\hom}{\mathrm{Hom}}

\renewcommand{\Im}{\mathrm{Im}} 
\newcommand\sr{\stackrel}
\newcommand\st{\scriptstyle}
\newcommand\sst{\scriptscriptstyle}
\newcommand\cGG{\check{\GG}}
\newcommand\hGG{\hat{\GG}}

\newcommand\hP{\hat{\Pi}}

\newcommand\ccC{\check{C}}

\newcommand\ssm{\smallsetminus}
\newcommand\ol{\overline}
\newcommand\ccM{\overline{\mathcal M}}
\newcommand\cM{{\mathcal M}}
\newcommand\cA{{\mathcal A}}
\newcommand\cN{{\mathcal N}}

\newcommand\cC{{\mathscr C}}
\newcommand\cE{{\mathscr E}}
\newcommand\tcC{\widetilde{\mathscr C}}
\newcommand\cD{{\cal D}}

\newcommand\cG{{\mathscr G}}
\newcommand\cS{{\mathscr S}}

\newcommand\GG{\Gamma}
\newcommand\ld{\lambda}
\newcommand\up{\upsilon}
\newcommand\Ld{\Lambda}

\newcommand\wt{\widetilde}

\newcommand\tcM{\wt{\cM}}

\newcommand\td{\tilde}
\newcommand\sg{\sigma}
\newcommand\Sg{\Sigma}
\newcommand\gm{\gamma}
\newcommand\bt{\bullet}

\def\co{\colon\thinspace}

\newtheorem{theorem}{Theorem}[section]
\newtheorem{corollary}[theorem]{Corollary}
\newtheorem{proposition}[theorem]{Proposition}
\newtheorem{lemma}[theorem]{Lemma}

\theoremstyle{definition}
\newtheorem{definition}[theorem]{Definition}   
\newtheorem{remark}[theorem]{Remark}

\begin{document}

\title{On the procongruence completion\\ of the Teichm\"uller modular group}
\author{Marco Boggi}\maketitle

\begin{abstract}
For $2g-2+n>0$, the Teichm{\"u}ller modular group $\GG_{g,n}$ of a compact Riemann
surface of genus $g$ with $n$ points removed $S_{g,n}$ is the group
of homotopy classes of diffeomorphisms of $S_{g,n}$ which preserve the
orientation of $S_{g,n}$ and a given order of its punctures. Let $\Pi_{g,n}$ be the fundamental
group of $S_{g,n}$, with a given base point, and $\hP_{g,n}$ its profinite completion.
There is then a natural faithful representation $\GG_{g,n}\hookra\out(\hP_{g,n})$. 
{\it The procongruence Teichm{\"u}ller group} $\cGG_{g,n}$ is defined to be the closure
of the Teichm{\"u}ller group $\GG_{g,n}$ inside the profinite group $\out(\hP_{g,n})$.

In this paper, we begin a systematic study of the procongruence completion $\cGG_{g,n}$. 
The set of {\it profinite Dehn twists} of $\cGG_{g,n}$ is the closure, inside this group, of the set 
of Dehn twists of $\GG_{g,n}$. The main technical result of the paper is a parametrization of the 
set of profinite Dehn twists of $\cGG_{g,n}$ and the subsequent description of their centralizers 
(\S~\ref{profinite twists} and \ref{centralizers}). This is the basis for the Grothendieck-Teichm\"uller
Lego with procongruence Teichm\"uller groups as building blocks.

As an application, in Section~\ref{Galois}, we prove that some Galois representations 
associated to hyperbolic curves over number fields and their moduli spaces are faithful. 
\newline

\end{abstract}

\section{Introduction}
Let $C$ be a hyperbolic curve defined over a number field $\K$. Let us fix an embedding 
$\K\subset\ol{\Q}$ and a $\ol{\Q}$-valued point $\tilde{\xi}\in C$. Then, the structural morphism
$C\ra\Spec(\K)$ induces a short exact sequence of algebraic fundamental groups:
$$1\ra\pi_1(C\times_\K\ol{\Q},\tilde{\xi})\ra\pi_1(C,\tilde{\xi})\ra G_\K\ra 1,$$
where the absolute Galois group $G_\K$ is identified with the algebraic 
fundamental group $\pi_1(\Spec(\K),\Spec(\ol{\Q}))$. 
Associated to the above short exact sequence, is the outer Galois representation:
$$\rho_C\co G_\K\ra\out(\pi_1(C\times_\K\ol{\Q},\tilde{\xi})).$$
This representation plays a central role in the arithmetic geometry of curves.

The representation $\rho_C$ can be interpreted as the monodromy representation associated to the 
curve $C\ra\Spec(\K)$ and the given choices of base points. Hence, it can be recovered from
the arithmetic universal monodromy representation as follows.
 
Let us suppose that $C$ is an $n$-punctured, genus $g$ curve, for $2g-2+n>0$.
Let then $\cM_{g,n}$ be the moduli stack of smooth $n$-pointed, genus $g$ curves
defined over some number field. It is a smooth irreducible Deligne--Mumford stack of dimension 
$3g-3+n$ defined over $\Spec(\Q)$, endowed with a universal $n$-punctured, genus $g$ curve 
$\cC\ra\cM_{g,n}$. Let $\xi\in\cM_{g,n}$ be the point corresponding to the curve $C$ and $\ol{\xi}$ 
the $\QQ$-valued point lying over it. Let us identify the fiber $\cC_{\ol{\xi}}$ with the curve 
$C\times_\K\QQ$. There is then a short exact sequence of algebraic fundamental groups:
$$1\ra\pi_1(C\times_\K\ol{\Q},\tilde{\xi})\ra\pi_1(\cC,\tilde{\xi})\ra\pi_1(\cM_{g,n},\ol{\xi})\ra 1.$$
The associated outer representation
$$\mu_{g,n}\co\pi_1(\cM_{g,n},\ol{\xi})\ra\out(\pi_1(C\times_\K\ol{\Q},\tilde{\xi}))$$
is {\it the arithmetic universal monodromy representation}.

The morphism of pointed stacks $\xi\co(\Spec(\K),\Spec(\QQ))\ra(\cM_{g,n},\ol{\xi})$
induces a homomorphism $\xi_\ast\co G_\K\ra\pi_1(\cM_{g,n},\ol{\xi})$ on algebraic fundamental 
groups. The representation $\rho_C$ is then obtained composing $\xi_\ast$ with the arithmetic 
universal monodromy representation $\mu_{g,n}$. Thus, it is not a surprise that the representation 
$\mu_{g,n}$ contains essential informations on the properties shared by all Galois outer 
representations associated to smooth $n$-punctured, genus $g$ arithmetic curves. The study of 
the representation $\mu_{g,n}$ is one of the principal purposes of this paper.

The structural morphism $\cM_{g,n}\ra\Spec(\Q)$ induces the short exact sequence:
$$1\ra\pi_1(\cM_{g,n}\times\ol{\Q},\ol{\xi})\ra\pi_1(\cM_{g,n},\ol{\xi})\ra G_\Q\ra 1.$$
The left term of this short exact sequence is called {\it the geometric algebraic
fundamental group of} $\cM_{g,n}$. The reason for this terminology is that
it is naturally isomorphic to the profinite completion of the topological fundamental group with base 
point $\ol{\xi}$ of the complex analytic stack associated to the complex Deligne-Mumford stack 
$\cM_{g,n}\times\C$ (cf. \cite{N1}, \cite{N2}).

Let $S_{g,n}$ be an $n$-punctured, genus $g$ Riemann surface and let $\GG_{g,n}$ be the 
associated Teichm\"uller modular group. The choice of a homeomorphism 
$\phi\co S_{g,n}\ra (C\times_\K\C)^{\mathrm{top}}$ determines an isomorphism between $\GG_{g,n}$ 
and $\pi_1^{\mathrm{top}}(\cM_{g,n}\times\C,\ol{\xi})$ (for more details, see Section~\ref{levels}). 
So that we get also an identification of the geometric algebraic fundamental group of $\cM_{g,n}$ 
with the profinite completion of the Teichm\"uller group $\GG_{g,n}$, which we simply call 
{\it the profinite Teichm\"uller group} $\hGG_{g,n}$.

Let us denote by $\Pi_{g,n}$ the topological fundamental group 
$\pi_1(S_{g,n},\phi^{-1}(\tilde{\xi}))$ and by $\hP_{g,n}$ its profinite completion. The restriction of 
the arithmetic universal monodromy representation $\mu_{g,n}$ to the geometric algebraic 
fundamental group then induces a representation:
$$\hat{\rho}_{g,n}\co\hGG_{g,n}\ra\out(\hP_{g,n}),$$
which we call {\it the profinite universal monodromy representation}. This can be also described
as the homomorphism of profinite groups induced by the canonical homotopy action of 
$\GG_{g,n}$ on $S_{g,n}$ (see Section~\ref{levels}).

It is clear that, in order to gain some insight on the representation $\mu_{g,n}$, we need to study 
before the representation $\hat{\rho}_{g,n}$. The study of the latter representation is the main 
subject of the paper from Section~\ref{levels} till Section~\ref{centralizers}.

The first question which arises about the representation $\hat{\rho}_{g,n}$ is whether or not it is
faithful. This is better known as {\it the congruence subgroup problem} for the
Teichm\"uller group $\GG_{g,n}$. An affirmative answer is known only in genus $\leq 2$
(cf. \cite{Asada}, \cite{Hyp2}). 

The investigation underlying this paper was stimulated by the idea that, in order to advance in
all the above issues, it was necessary a complete understanding of the group-theoretic
properties of the image of the representation $\hat{\rho}_{g,n}$, which we denote
by $\cGG_{g,n}$ and call {\it the procongruence Teichm\"uller group}.

The combinatorial group theory of the Teichm\"uller group $\GG_{g,n}$ begins with the study of the
relations occurring between words in its standard set of generators which are the Dehn twists.
{\it The set of profinite Dehn twists} of the procongruence Teichm\"uller group $\cGG_{g,n}$ is the 
closure of the image of the set of Dehn twists via the natural monomorphism 
$\GG_{g,n}\hookra\cGG_{g,n}$. 

The key technical result of this paper is the characterization of powers of profinite twists
in $\cGG_{g,n}$ which is given in Theorem~\ref{parametrize}. As an immediate consequence, we 
get a description of the centralizers and of the normalizers of the closed abelian subgroups 
spannned by set of powers of profinite twists in $\cGG_{g,n}$, in perfect analogy with the classical 
results in $\GG_{g,n}$ (see Corollary~\ref{centralizer2} and Theorem~\ref{normalizer}). 

Deeper and more involved relations occurring between Dehn twists in $\GG_{g,n}$ are 
encoded in the various curve complexes which can be associated to $S_{g,n}$. Here, we mention
just the most important one: the complex of curves $C(S_{g,n})$. This is the simplicial complex whose
simplices are given by sets of distinct, non-trivial, isotopy classes of simple closed curves (briefly 
s.c.c.) on $S_{g,n}$, such that they admit a set of disjoint representatives none of them bounding a 
disc with a single puncture. This complex is endowed with a natural simplicial action of the 
Teichm\"uller group $\GG_{g,n}$. 


Sections~\ref{completions} and \ref{geocomplex} are devoted 
to construct a satisfactory profinite analogue of the curve complex $C(S_{g,n})$. As the final result of 
our efforts, we get a simplicial profinite complex $L(\hP_{g,n})$ which we call {\it the complex
of profinite curves}. Eventually, in Theorem~\ref{intrinsic}, this is characterized as the simplicial
complex whose $k$-simplices are the closed abelian subgroups
of rank $k+1$ spanned by profinite Dehn twists in $\cGG_{g,n}$. 

The importance of the complex of profinite curves is well illustrated by Corollary~7.3 \cite{sym},
where it is shown that the obstruction for the congruence subgroup property to hold in genus greater 
than $2$ is the profinite fundamental group of $L(\hP_{g})$, for $g>2$.

All the results mentioned above then form the basis for the study of Galois representations associated
to hyperbolic curves carried out in Section~\ref{Galois}. The main result is that the outer 
representation $\rho_C$ associated to a hyperbolic curve $C$ over a number field $\K$ is faithful 
(Theorem~\ref{outer-galois}). The faithfulness of the arithmetic universal monodromy
representation $\mu_{g,n}$ is instead reduced to the congruence subgroup problem for $\GG_{g,n}$. 
In particular, it follows that $\mu_{g,n}$ is faithful for $g\leq 2$ (Corollary~\ref{universal-faithful}). 
The proofs of all these statements, especially Theorem~\ref{outer}, are inspired by the ideas of 
Grothendieck-Teichm\"uller theory as they appeared in the \textsl{Esquisse d'un Programme} 
\cite{Esquisse}. With a substantially different approach, similar results have been proved
in \cite{Matsu} and \cite{H-M}.

\section{Level structures over moduli of curves}\label{levels}

The study we are going to carry out from Section~\ref{levels} to Section~\ref{centralizers} is in 
its essence topological. Hence, in order to avoid cumbersome notations, for a complex 
Deligne-Mumford (briefly D--M) stack $X$, we will denote
by $\pi_1(X)$ its topological fundamental group and by $\hat{\pi}_1(X)$ its algebraic fundamental
group (when unnecessary, we omit to mention base-points). 
This notation is consistent with the fact that $\hat{\pi}_1(X)$ is isomorphic to the profinite
completion of $\pi_1(X)$. For the same reasons, $\ccM_{g,n}$, for $2g-2+n>0$, 
will denote the stack of $n$--pointed, genus $g$, stable complex algebraic
curves. It is a smooth irreducible proper complex D--M stack of dimension $3g-3+n$, and it 
contains, as an open substack, the stack $\cM_{g,n}$ of $n$--pointed, genus $g$, smooth complex 
algebraic curves. We will keep the same notations to denote the respective underlying analytic 
and topological stacks. 

The stack $\ccM_{g,n}$ is simply connected. On the contrary, 
the stack $\cM_{g,n}$ has plenty of non-trivial coverings which we are briefly going to introduce in 
this section. We refer to \cite{sym} for more details and references on all the following constructions. 

The universal cover of $\cM_{g,n}$ is the Teichm{\"u}ller space $T_{g,n}$. This is the stack of 
$n$-pointed, genus $g$, smooth 
complex analytic curves $(\cE\ra{\mathscr U}, s_1,\ldots, s_n)$ endowed with a topological 
trivialization $\Phi\co S_{g,n}\times{\mathscr U}\sr{\sim}{\ra} \cE\ssm\cup_{i=1}^n s_i({\mathscr U})$ 
over ${\mathscr U}$, where two such trivializations are considered equivalent when they are 
homotopic over ${\mathscr U}$. We then denote the corresponding object of $T_{g,n}$ by 
$(\cE\ra{\mathscr U}, s_1,\ldots, s_n,\Phi)$ or, when ${\mathscr U}$ is just one point, simply 
by $(E,\Phi)$. The complex analytic stack $T_{g,n}$ is represented by a contractible complex 
manifold. Then, the natural map of complex analytic stacks $T_{g,n}\ra\cM_{g,n}$ is a universal 
cover. The deck transformation group of this covering is described as follows. 

Let $\hom^+(S_{g,n})$ be the group of orientation preserving self-homeomorphisms
of $S_{g,n}$ and by $\hom^0(S_{g,n})$ the subgroup consisting of homeomorphisms homotopic to 
the identity. {\it The mapping class group} or {\it Teichm\"uller modular group} $\GG_{g,[n]}$
is classically defined to be the group of homotopy classes of homeomorphisms of $S_{g,n}$ 
which preserve the orientation:
$$\GG_{g,[n]}:=\left.\hom^+(S_{g,n})\right/\hom^0(S_{g,n}),$$  
where $\hom_0(S_{g,n})$ is the connected component of the identity 
in the topological group of homeomorphisms $\hom^+(S_{g,n})$. 
This group then is the group of deck transformation of the \'etale covering $T_{g,n}\ra\cM_{g,[n]}$, 
where the latter denotes  the stack of genus $g$, smooth complex curves with $n$ unordered 
markings. Therefore, there is a short exact sequence:
$$1\ra\GG_{g,n}\ra\GG_{g,[n]}\ra\Sigma_n\ra 1,$$
where $\Sigma_n$ denotes the symmetric group on the set of punctures of $S_{g,n}$
and $\GG_{g,n}$ is {\it the pure mapping class group} or {\it pure Teichm\"uller modular group} 
associated to $S_{g,n}$ and the deck transformation group of the covering $T_{g,n}\ra\cM_{g,n}$. 

There is a natural way to define homotopy groups for topological D--M stacks (cf. \cite{N2}), so that 
the choice of a point $a=[C]\in\cM_{g,n}$ and a homeomorphism $\phi\co S_{g,n}\ra C$ identifies 
the topological fundamental groups $\pi_1(\cM_{g,n},a)$ and $\pi_1(\cM_{g,[n]},a)$ with the pure 
mapping class group $\GG_{g,n}$ and with the mapping class group $\GG_{g,[n]}$, respectively.

Let $p\co\cC\ra\cM_{g,[n]}$ be the universal $n$-punctured, genus $g$ curve. 
Since $p$ is a Serre fibration and $\pi_2(\cM_{g,[n]})=\pi_2(T_{g,n})=0$, there is a short exact 
sequence
$$1\ra\pi_1(\cC_a,\tilde{a})\ra\pi_1(\cC,\tilde{a})\ra\pi_1(\cM_{g,[n]},a)\ra 1,$$
where $\tilde{a}$ is a point in the fiber $\cC_a$.
By a standard argument this defines a monodromy representation:
$$\rho_{g,[n]}\co\pi_1(\cM_{g,[n]},a)\ra \mbox{Out}(\pi_1(\cC_a,\tilde{a})),$$
called {\it the universal monodromy representation}. We then denote by $\rho_{g,n}$ its
restriction to the pure mapping class group $\GG_{g,n}$ and call it the same way.

Let us then fix a homeomorphism $\phi\co S_{g,n}\ra\cC_a$ and let
$\Pi_{g,n}$ be the fundamental group of $S_{g,n}$ based in
$\phi^{-1}(\tilde{a})$. Then, the representation $\rho_{g,[n]}$ is identified with
the faithful representation $\GG_{g,[n]}\hookra\out(\Pi_{g,n})$, induced by the 
homotopy action of $\GG_{g,[n]}$ on the Riemann surface $S_{g,n}$.

In this paper, a level structure $\cM^\ld$ is a finite, connected, Galois, {\'e}tale covering of the stack 
$\cM_{g,[n]}$ (by {\'e}tale covering, we mean here an {\'e}tale, surjective, representable morphism of
algebraic stacks), therefore it is also represented by a smooth D--M stack $\cM^\ld$.  This is
in contrast with \cite{PFT}, where the covering was not required to be Galois.
{\it The level} associated to $\cM^\ld$ is the finite index normal subgroup
$\GG^\ld:=\pi_1(\cM^\ld, a')$ of the Teichm{\"u}ller group $\GG_{g,[n]}$.

For a given level $\GG^\ld$ of $\GG_{g,[n]}$, the intersection $\GG^\ld\cap\GG_{g,n}$ is also
a level of $\GG_{g,[n]}$ which we denote by $\GG^\ld_{g,n}$. The corresponding level structure is
denoted by $\cM_{g,n}^\ld$. Equivalently, the level structure $\cM_{g,n}^\ld$ is the pull-back over 
$\cM_{g,n}\ra\cM_{g,[n]}$ of the level structure $\cM^\ld$.

The most natural way to define levels is by means of the universal monodromy representation 
$\rho_{g,[n]}$. For a characteristic subgroup $K\le\Pi_{g,n}$, let us define the representation:
$$\rho_K\co\GG_{g,[n]}\ra \mbox{Out}(\Pi_{g,n}/K),$$
whose kernel we denote by $\GG^K$. When $K$ has finite index in $\Pi_{g,n}$, then
$\GG^K$ has finite index in $\GG_{g,[n]}$ and is called {\it the geometric level} associated to
$K$. The corresponding level structure is denoted by $\cM^{K}$. 

Of particular interest are the levels defined by the kernels of the representations:
$$\rho_{(m)}\co\GG_{g,[n]}\ra \mbox{Sp}(H_1(S_g,\Z/m)),\;\;\mbox{ for }m\geq 2.$$
They are denoted by $\GG(m)$ and called {\it abelian levels of order $m$}. The corresponding
level structures are then denoted by $\cM^{(m)}$. A classical result of Serre says that
an automorphism of a smooth curve acting trivially on its first homology group with $\Z/m$
coefficients, for $m\geq 3$, is trivial. This implies that any level structure
$\cM^\ld$ dominating an abelian level structure $\cM^{(m)}$, with $m\geq 3$, is 
representable in the category of algebraic varieties.

Another way to define levels, introduced in \cite{sym}, will be extremely useful here. Let us 
briefly recall this construction, referring for more details to \cite{sym}.

Let $K$ be a proper normal finite index subgroup of $\Pi_{g,n}$, which is invariant for the action
of $\GG_{g,[n]}$, and let $p_K\co S_K\ra S_{g,n}$ be 
the \'etale Galois covering with deck transformation group $G_K$, associated to such subgroup.  
Let us denote by $\GG(S_K)$ the orientation preserving mapping class group of the Riemann 
surface $S_K$. There are then a natural monomorphism $G_K\hookra\GG(S_K)$ and a natural
epimorphism from the normalizer $N_{\GG(S_K)}(G_K)$ of $G_K$ in $\GG(S_K)$ to the mapping
class group $\GG_{g,[n]}$. So, there is a natural short exact sequence:
$$1\ra G_K\ra N_{\GG(S_K)}(G_K)\ra\GG_{g,[n]}\ra 1.\hspace{1cm}(2.1)$$

From Hurwitz's Theorem and the fact that $K$ is a proper invariant subgroup of $\Pi_{g,n}$, it follows 
that, for $g\geq 1$ or $n\geq 4$, the compact Riemann surface $\ol{S}_K$ obtained filling in the 
punctures of $S_K$ has genus at least one. For $m\geq 2$, let us then consider the natural 
representation $\rho_{(m)}\co\GG(S_K)\ra\Sp(H_1(\ol{S}_K,\Z/m))$.

For $m\geq 3$ and $m=0$, the restriction of $\rho_{(m)}$ to $G_K$ is faithful. For $m=2$, the 
restriction $\rho_{(m)}|_{G_K}$ is faithful unless $G_K$ contains a hyperelliptic involution, in which 
case the hyperelliptic involution generates the kernel of $\rho_{(m)}|_{G_K}$ (cf. Lemma~2.9, 
Ch. XVII, \cite{A-C}). In particular, for $g\geq 1$ or, for $g=0$, if $[\Pi_{0,n}:K]>2$, the group $G_K$ 
does not contain a hyperelliptic involution and so the restriction of $\rho_{(2)}$ to $G_K$ is faithful 
as well. 

Let us assume that the restriction $\rho_{(m)}|_{G_K}$ is faithful and let us denote by $G_K$  
its image. For $m\geq 2$, there is a natural representation:
$$\rho_{K,(m)}\co\GG_{g,[n]}\ra\left. N_{\Sp(H_1(\ol{S}_K,\Z/m))}(G_K)\right/G_K.$$
Let us denote the kernel of $\rho_{K,(m)}$ by $\GG^{K,(m)}$ and call it {\it the Looijenga level} 
associated to the subgroup $K$ of $\Pi_{g,n}$. 
The corresponding level structure is denoted by $\cM^{K,(m)}$.

Geometric levels can also be described in terms of the exact sequence $(2.1)$.
The geometric level $\GG^K$ associated to $K$ is indeed the set of elements of $\GG_{g,[n]}$ 
which admit a lift, by the natural epimorphism $N_{\GG(S_K)}(G_K)\tura\GG_{g,[n]}$,  to 
the centralizer $Z_{\GG(S_K)}(G_K)$. Thus, there is a short exact sequence:
$$1\ra Z(G_K)\ra Z_{\GG(S_K)}(G_K)\ra\GG^K\ra 1,\hspace{1cm}(2.2)$$
where $Z(G_K)$ is the center of the group $G_K$. If this center is trivial, the geometric level 
$\GG^K$ is naturally identified with the centralizer of $G_K$ inside $\GG(S_K)$.

A key fact is that the tower of Loojenga levels is equivalent to the tower
of geometric levels. More precisely, it holds (Theorem~2.2 and Corollary~2.3 in \cite{sym}):

\begin{theorem}\label{comparison}\begin{enumerate}
\item For $2g-2+n>0$, let $K$ be a finite index invariant subgroup 
of $\Pi_{g,n}$. Then, for $m\geq 3$, or, if the quotient group $G_K$ does not contain a hyperelliptic 
involution, for $m\geq 2$, there is an inclusion of levels $\GG^{K,(m)}\unlhd\GG^K$. 

\item For any fixed integer $m\geq 2$, the set of Looijenga levels 
$\{\GG^{K,(m)}\}_{K\unlhd\Pi_{g,n}}$ forms an inverse system of finite index normal subgroups of 
$\GG_{g,[n]}$ which defines the same profinite topology as the tower of geometric levels 
$\{\GG^K\}_{K\unlhd\Pi_{g,n}}$.
\end{enumerate}
\end{theorem}

The usual way to compactify a level structure $\cM^{\ld}$ over $\cM_{g,[n]}$ is to take the 
normalization of $\ccM_{g,[n]}$ in the function field of $\cM^{\ld}$. A more functorial definition can be 
given in the category of regular log schemes. Indeed, it is easy to see that the natural morphism of 
logarithmic stacks $(\ccM^{\ld})^{log}\ra\ccM_{g,[n]}^{log}$ is log-\'etale, where $(\_)^{log}$ denotes
the logarithmic structure associated to the respective D--M boundaries 
$\dd\cM^\ld:=\ccM^\ld \ssm\cM^\ld$ and $\dd\cM_{g,[n]} :=\ccM_{g,[n]} \ssm\cM_{g,[n]}$. 
Viceversa, by the log purity Theorem, any finite, connected, log {\'e}tale Galois covering of
$\ccM_{g,[n]}^{log}$ is of the above type. 

A basic property of (compactified) level structures, proved by Brylinski and Deligne (cf. (ii),
Proposition~2 \cite{Br} and Proposition~3.2 \cite{sym}), is that 
if the level $\GG^\lambda$ is contained in an abelian level of order $m$, for some $m\ge 3$, then 
the level structure $\ccM^{\ld}_{g,n}$ is represented by a projective variety.

An analogue of the D--M compactification $\ccM_{g,n}$ of $\cM_{g,n}$, at the level of 
Teichm\"uller space, is the Bers bordification $\ol{T}_{g,n}$ of the Teichm\"uller space 
$T_{g,n}$, which we are shortly going to describe
(see \S 3, Ch. II, \cite{Abikoff} for more details on this construction). 

For a given stable $n$-pointed, genus $g$ curve $C$, let $\cN$ be its singular set and $\cal P$ its
set of marked points. {\it A degenerate marking} $\phi\co S_{g,n}\ra C$ is a continuous 
map such that $C\ssm\Im\,\phi={\cal P}$, the inverse image $\phi^{-1}(x)$, for all $x\in\cN$, is 
a simple closed curve (briefly, s.c.c.) on $S_{g,n}$ and the restriction of the marking 
$\phi\co S_{g,n}\ssm\phi^{-1}(\cN)\ra C\ssm(\cN\cup {\cal P})$ 
is a homeomorphism. {\it The Bers bordification} $\ol{T}_{g,n}$ of the Teichm\"uller space 
$T_{g,n}$ is the real analytic space which parametrizes pairs $(C,\phi)$, consisting 
of a stable $n$-pointed, genus $g$ curve $C$ and the homotopy class
of a degenerate marking $\phi\co S_{g,n}\ra C$.

The natural action of $\GG_{g,[n]}$ on $T_{g,n}$ extends to $\ol{T}_{g,n}$. However, this action
is not anymore proper and discontinuous, since a boundary stratum has for inertia group the free
abelian group generated by the Dehn twists along the simple closed curves of $S_{g,n}$ which are
collapsed on such boundary stratum.

The geometric quotient $\ol{T}_{g,n}/\GG_{g,[n]}$ identifies with the real-analytic space underlying
the coarse moduli space $\ol{M}_{g,[n]}$ of stable $n$-pointed, genus $g$ curves.
On the other hand, the quotient stack $[\ol{T}_{g,n}/\GG_{g,[n]}]$ only admits a non-representable 
natural map to the D--M stack $\ccM_{g,[n]}$, because of the extra-inertia at infinity. 

The irreducible components of the boundary $\dd T_{g,n}:=\ol{T}_{g,n}\ssm T_{g,n}$ are
isomorphic either to $\ol{T}_{g-1,n+2}$ or to $\ol{T}_{g_1,n_1+1} \times\ol{T}_{g_2,n_2+1}$,
for some non-negative integers $n_1$, $n_2$ and $g_1$, $g_2$ such that $n=n_1+n_2$ and 
$g=g_1+g_2$, with the condition that $n_i \geq 2$ when $g_i =0$. 

The nerve of the cover of the boundary $\dd T_{g,n}$ by its irreducible components is explicitly
described by a simplicial complex called the complex of curves and denoted by
$C(S_{g,n})$. It is the simplicial complex whose simplices are sets of distinct, non-trivial 
isotopy classes of s.c.c.'s on $S_{g,n}$, such that they admit a set of disjoint representatives, 
none of them bounding a disc with a single puncture. 

The above description of the nerve of the boundary of $\ol{T}_{g,n}$ can be extended to level
structures satisfying some mild hypotheses.
For a simplex $\sg\in C(S_{g,n})$, let us denote by $\GG_\sg$ the stabilizer of 
$\sg$ for the natural action of the Teichm\"uller group $\GG_{g,n}$
on $C(S_{g,n})$ and by $\GG_{\vec{\sg}}$ the subgroup of elements of $\GG_\sg$ which 
fix all the s.c.c.'s in the simplex $\sg$, preserving moreover their orientations. It then holds
(Proposition~4.2 \cite{sym}):

\begin{proposition}\label{without inversions}For $2g-2+n>0$, let $\GG^\ld$ be a level of 
$\GG_{g,n}$ contained in an abelian level $\GG(m)$, for $m\geq 3$. Then, it holds 
$\GG^\ld\cap \GG_{\vec{\sg}}=\GG^\ld\cap \GG_\sg$, i.e. the level $\GG^\ld$
operates without inversions on the curve complex $C(S_{g,n})$
\end{proposition}

For the level structures $\ccM^\ld$ such that $\GG^\ld$ is contained in an abelian level $\GG(m)$, 
for some $m\geq 3$, there is then a direct way to realize the nerve of the cover of 
their D--M boundary by irreducible components in the category of simplicial sets. 

\begin{definition}\label{finite complex}
Let $C(S_{g,n})_\bt$ be the simplicial set associated to the simplicial complex
$C(S_{g,n})$ and an ordering of its vertices compatible with the action of a level $\GG^\ld$ 
contained in an abelian level $\GG(m)$, for some $m\geq 3$. Then, the simplicial finite set 
$C^\ld(S_{g,n})_\bt$ is defined to be the quotient of $C(S_{g,n})_\bt$ by the simplicial action 
of $\GG^\ld$. This simplicial set describes the nerve of the 
cover of $\dd\cM^\ld$ by its irreducible components.
\end{definition}

\section{Profinite completions}\label{completions}

For $2g-2+n>0$, there is a short exact sequence
$1\ra\hP_{g,n}\ra\hGG_{g,n+1}\ra\hGG_{g,n}\ra 1$.
The action induced by the restriction of inner automorphisms of $\hGG_{g,n+1}$ to 
its normal subgroup $\hP_{g,n}$ defines a representation 
$\tilde{\rho}_{g,n}\co\hGG_{g,n+1}\ra\mbox{Aut}(\hat{\Pi}_{g,n})$, which then induces 
the profinite universal monodromy representation 
$\hat{\rho}_{g,n}\co\hGG_{g,n}\ra\mbox{Out}(\hat{\Pi}_{g,n})$.

\begin{definition}\label{geometric}For $2g-2+n>0$, let us define the profinite group
$\cGG_{g,n}$ to be the image of $\hat{\rho}_{g,n}$ in $\mbox{Out}(\hat{\Pi}_{g,n})$.
We then call $\cGG_{g,n}$ {\it the procongruence completion of the Teichm\"uller 
group} or, more simply, {\it the procongruence Teichm\"uller group}.
\end{definition}
 
By definition, there is a natural continuous homomorphism $i\co\GG_{g,n}\ra\cGG_{g,n}$, with 
dense image. That this map is injective is instead a deeper result by Grossman \cite{Grossman}.

By Theorem~2.4 \cite{Hyp2}, for $2g-2+n>0$, there is a natural isomorphism of profinite groups 
$\Im\hat{\rho}_{g,n+1}\equiv\Im\tilde{\rho}_{g,n+1}$ and then a natural short exact sequence:
$$1\ra\hP_{g,n}\ra\cGG_{g,n+1}\ra\cGG_{g,n}\ra 1.$$

We will now associate to the procongruence completion of the Teichm\"uller group introduced 
above a completion of the curve complex. This, however, requires a preliminary digression.
Let us recall some definitions from \cite{PFT}.

Let $X_\bt$ be a simplicial set and $G$ a group. {\it A geometric action} of $G$ on $X_\bt$ is an 
action of $G$ on $X_n$, defined for all $n\geq 0$, satisfying the following conditions:
\begin{enumerate}
\item For all $g\in G$ and $\sg\in X_n$, one has 
$\{\dd_i\, g\!\cdot\!\sg|i=0,\ldots, n\}=\{g\cdot\dd_i\sg|i=0,\ldots, n\}$. In this way, are defined 
permutation representations $\rho_\sg: G\ra\Sigma_{n+1}$, which we require to be constant on 
$G$-orbits, i.e. $\rho_\sg(g)=\rho_{g'\cdot\sg}(g)$, for all $g,g'\in G$.
\item By the above condition, an action of $G$ on the set $\coprod_{n\geq 0}X_n\times\Delta_n$ is 
defined, for $(\sg, x)\in X_n\times\Delta_n$ and $g\in G$, letting:
$g\cdot (\sg, x):=(g\cdot\sg, \rho_\sg (g)(x))$.
We then require that this action be compatible with the equivalence relation "$\sim$" which
defines the geometric realization $|X_\bt|:=\coprod\limits_{n\geq 0}\left.X_n\times\Delta_n\right/\!\sim$.
\end{enumerate}

If, for $n\geq 0$, the actions $G\times X_n\ra X_n$ commute with the face and degeneracy 
operators, then $X_\bt$ is said to be {\it a simplicial $G$-set}. In this case the quotient sets 
$X_n/G$, for $n\geq 0$, together with the induced face and degeneracy operators, form a 
simplicial set $X_\bt/G$. Moreover, since the representations $\rho_\sg$ are trivial for all 
$\sg\in X_\bt$, there is a natural isomorphism $|X_\bt|/G\cong |X_\bt/G|$. 

Let $G$ be a profinite group and $X_\bt$ a simplicial profinite set (i.e. a simplicial object in the
category of profinite sets and their continuous maps). We define a continuous geometric action 
of $G$ on $X_\bt$ to be a geometric action such that each action $G\times X_n\ra X_n$ is 
continuous and with open orbits, for $n\geq 0$. Moreover, we require that there exists an open 
subgroup $U\leq G$ such that $X_\bt$, with the induced $U$-action, is a simplicial $U$-set. 

Let $\{G^\ld\}_{\ld\in\Ld}$ be the tower of open subgroups of $G$ contained in $U$. Then, for all
$\ld\in\Ld$, the quotient $X_\bt/G^\ld$ is a simplicial discrete finite set endowed with a geometric
continuous action of $G$.

Let now $G\ra G'$ be a homomorphism of a discrete group in a profinite group, with dense image, 
and let $\{G^\ld\}_{\ld\in\Ld}$ be the tower of sugroups of $G$ which are inverse images of open 
subgroups of $G'$. Let $X_\bt$ be a simplicial set endowed with a geometric $G$-action such that 
the set of $G$-orbits in $X_n$ is finite for all $n\geq 0$. Assume, moreover, that there exists a 
$\mu\in\Ld$ such that $X_\bt$, with the induced $G^\mu$-action, is a simplicial $G^\mu$-set. 

The $G'$-completion of $X_\bt$ is defined to be the simplicial profinite set
$$X_\bt':=\lim\limits_{\sr{\textstyle\st\longleftarrow}{\sst\ld\in\Ld}}X_\bt/G^\ld,$$
which is then endowed with a natural continuous geometric action of $G'$. The simplicial profinite 
set $X'_\bt$ has the following universal property. Let $f\co X_\bt\ra Y_\bt$ be a simplicial 
$G$-equivariant map, where $Y_\bt$ is a simplicial profinite set endowed with a continuous 
geometric action of $G'$. Then $f$ factors uniquely through the natural map
$X_\bt\ra X_\bt'$ and a simplicial $G'$-equivariant continuous map $f'\co X_\bt'\ra Y_\bt$.

Let us apply the definition of profinite $G'$-completion of a simplicial set endowed with a geometric 
$G$-action to the the Teichm{\"u}ller group $\GG_{g,n}$ and its procongruence completion 
$\cGG_{g,n}$. The tower of levels which defines the procongruence completion contains the
abelian level $\GG(m)$, for $m\geq 3$. By Proposition~\ref{without inversions}, 
there is then an ordering of the vertex set of the simplicial complex $C(S_{g,n})$ such that
the associated simplicial set $C(S_{g,n})_\bt$ is a $\GG(m)$-simplicial set endowed with a natural 
geometric $\GG_{g,n}$-action.  For this ordering, the conditions prescribed in order to define
the $\cGG_{g,n}$-completion of the simplicial set $C(S_{g,n})_\bt$ are all satisfied. 

\begin{definition}\label{geocompletion}\begin{enumerate}
\item For $2g-2+n>0$, the simplicial profinite set 
$\ccC(S_{g,n})_\bt$ is defined to be the $\cGG_{g,n}$-completion of the simplicial set 
$C(S_{g,n})_\bt$. 
\item Let us define {\it a simplicial profinite complex} to be an abstract simplicial complex whose 
set of vertices is endowed with a profinite topology such that the sets of $k$-simplices, 
with the induced topologies, are compact and then profinite, for all $k\geq 0$. For these simplicial 
complexes, the procedure which associates to an abstract simplicial complex and an ordering of
its vertex set a simplicial set produces a simplicial profinite set. In Remarks~7.6 and 7.11 \cite{sym}, 
it is shown that $\ccC(S_{g,n})_\bt$ is the simplicial profinite set associated to a simplicial profinite 
complex $\ccC(S_{g,n})$, which we call {\it the procongruence curve complex}.
\end{enumerate}
\end{definition}

There is a natural $\GG_{g,n}$-equivariant map of simplicial complexes 
$i_\ast\co C(S_{g,n})\ra \ccC(S_{g,n})$. From Proposition~5.1 \cite{PFT} and the
fact that the map $i\co\GG_{g,n}\ra\cGG_{g,n}$ is injective, it follows:

\begin{proposition}\label{injectivity}For $2g-2+n>0$, the natural 
map $i_\ast\co C(S_{g,n})\ra \ccC(S_{g,n})$ is injective.
\end{proposition}

By Proposition~6.5 \cite{PFT}, the stabilizers for the action of an open subgroup of the 
procongruence Teichm{\"u}ller group $\cGG_{g,n}$ on the procongruence curve complex 
$\ccC(S_{g,n})$ are also determined:

\begin{proposition}\label{stab=clos}For $2g-2+n>0$, let $i\co\GG_{g,n}\hookra\cGG_{g,n}$ be the
natural monomorphism, let $\cGG^\ld$ be an open subgroup of $\cGG_{g,n}$ and let 
$\GG^\ld:=i^{-1}(\cGG^\ld)$. Then, the stabilizer $\cGG^\ld_\sg$ of a simplex 
$\sg\in \Im\,C(S_{g,n})\subset C'(S_{g,n})$ for the action of the subgroup $\cGG^\ld$ is the 
closure in $\cGG_{g,n}$ of the image of the stabilizer $\GG^\ld_\sg$ for the action of 
$\GG^\ld$ on $C(S_{g,n})$.
\end{proposition}

\section{The complex of profinite curves}\label{geocomplex}
The purpose of this section is to provide an alternative, more intrinsic, description of the
simplicial profinite complex $\ccC(S_{g,n})$, in terms of 
the profinite fundamental group $\hP_{g,n}$ of the Riemann surface $S_{g,n}$.

We say that a simple closed curve (s.c.c.) $\gm$ on the Riemann surface $S_{g,n}$
is {\it non-peripheral} if it does not bound a disc with less than two punctures.

Let $\cL_{g,n}\cong C(S_{g,n})_0$, for $2g-2+n>0$, denote the set of isotopy classes of 
non-peripheral simple closed curves on $S_{g,n}$. Let $\Pi_{g,n}/\!\sim$ be the set of conjugacy 
classes of elements of $\Pi_{g,n}$ and
$\cP_2(\Pi_{g,n}/\!\sim)$ the set of unordered pairs of elements of $\Pi_{g,n}/\!\sim$. 

For a given $\gm\in\Pi_{g,n}$, let us denote by $\gm^{\pm 1}$ the set $\{\gm,\gm^{-1}\}$ 
and by $[\gm^{\pm 1}]$ its equivalence class in $\cP_2(\Pi_{g,n}/\!\sim)$. There is then a 
natural embedding $\iota\co\cL_{g,n}\hookra\cP_2(\Pi_{g,n}/\!\sim)$, defined choosing, for 
an element $\gm\in\cL_{g,n}$, an element $\vec{\gm}_\ast\in\Pi_{g,n}$ whose free homotopy class 
contains $\gm$ and letting $\iota(\gm):=[\vec{\gm}_\ast^{\pm 1}]$.

Let $\hP_{g,n}/\!\sim$ be the set of conjugacy classes of elements of $\hP_{g,n}$ and
$\cP_2(\hP_{g,n}/\!\sim)$ the profinite set of unordered pairs of elements of $\hP_{g,n}/\!\sim$. 
From combinatorial group theory (cf. \cite{Stebe}), it is known that the set $\Pi_{g,n}/\!\sim$ 
embeds in the profinite set $\hP_{g,n}/\!\sim$. So, let us define the set of {\it non-peripheral 
profinite s.c.c.'s} $\hL_{g,n}$ on $S_{g,n}$ to be the closure of the set $\iota(\cL_{g,n})$ inside the 
profinite set $\cP_2(\hP_{g,n}/\!\sim)$. When it is clear from the context, we omit the subscripts and 
denote these sets simply by $\cL$ and $\hL$. 

There is a slight abuse of terminology here because the elements of $\cL$ are {\it isotopy classes 
of non-peripheral s.c.c.'s}. Possibly, it would have been more
appropriate to call the elements of $\hL$ {\it isotopy classes of non-peripheral profinite s.c.c.'s}.
This terminology could have been, at least partially, justified in terms of the profinite completion
of the Riemann surface $S_{g,n}$ but it seemed a bit too awkward and, in any case, not entirely 
meaningful.

The profinite set $\hL$ consists of equivalence classes $[\alpha^{\pm 1}]$ of pairs of 
elements of $\hP_{g,n}$, such that $\alpha$ is the limit of a sequence in $\Pi_{g,n}$ of 
elements representable by s.c.c.'s on $S_{g,n}$.

An ordering of the set $\{\alpha,\alpha^{-1}\}$ is preserved by the conjugacy action and 
defines an {\it orientation} for the associated equivalence class $[\alpha^{\pm1}]\in\hL$.

An alternative group-theoretic description of the set of non-peripheral s.c.c.'s $\cL$ is the following. 
Let $\cS(\Pi_{g,n})/\!\sim$ be the set of conjugacy classes of cyclic subgroups of $\Pi_{g,n}$. There is 
a natural embedding $\iota'\co\cL\hookra\cS(\Pi_{g,n})/\!\sim$, defined sending $\gm\in\cL$ to the 
conjugacy class of the cyclic subgroup of $\Pi_{g,n}$ generated by an element 
$\vec{\gm}_\ast\in\Pi_{g,n}$ whose free homotopy class contains $\gm$. So, let 
$\hat{\cS}(\hP_{g,n})/\!\sim$ be the profinite set of conjugacy classes of closed cyclic subgroups 
of $\hP_{g,n}$. There is a natural embedding of $\cS(\Pi_{g,n})/\!\sim$ in $\hat{\cS}(\hP_{g,n})/\!\sim$ 
and we define $\hL'$ to be the closure of $\iota'(\cL)$ inside $\hat{\cS}(\hP_{g,n})/\!\sim$. There is a 
natural continuous surjective map $\hL\tura\hL'$. However, it is not yet clear whether 
this map is also injective.

For all $k\geq 0$, there are natural embeddings of the sets of $k$-simplices of the
curve complex $C(S_{g,n})$ into the profinite sets $\cP_{k+1}(\hL)$ and
$\cP_{k+1}(\hL')$ of unordered subsets of $k+1$ elements of $\hL$ and $\hL'$, respectively. 
Let us then define:

\begin{definition}\label{geopro}Let $L(\hP_{g,n})$ and $L'(\hP_{g,n})$, for $2g-2+n>0$, 
be the simplicial profinite complexes whose set of $k$-simplices is the closure of the set of 
$k$-simplices of the curve complex $C(S_{g,n})$ inside, respectively, the profinite sets 
$\cP_{k+1}(\hL)$ and $\cP_{k+1}(\hL')$, for $k\geq 0$. 
The simplicial profinite complex $L(\hP_{g,n})$ is called 
{\it the complex of profinite curves on $S_{g,n}$}.
\end{definition}

By their definition, the simplicial profinite complexes defined above admit a natural continuous 
action of the procongruence Teichm\"uller group $\cGG_{g,n}$ and are 
linked by a natural continuous $\cGG_{g,n}$-equivariant surjective map 
$L(\hP_{g,n})\tura L'(\hP_{g,n})$.

The embeddings $C(S_{g,n})\hookra L(\hP_{g,n})$ and $C(S_{g,n})\hookra L'(\hP_{g,n})$ 
have dense images  and are $\GG_{g,n}$-equivariant. In particular, there are only finitely many
$\cGG_{g,n}$-orbits of $k$-simplices in $L(\hP_{g,n})$ and $L'(\hP_{g,n})$, for all $k\geq 0$, and 
each of them contains a discrete representative. 

By the universal property of the $\cGG_{g,n}$-completion, for every $k\geq 0$,  
there is a series of natural surjective continuous $\cGG_{g,n}$-equivariant maps of profinite sets:
$$\ccC(S_{g,n})_k\tura L(\hP_{g,n})_k\tura L'(\hP_{g,n})_k.$$
These maps are clearly compatible with the respective structures of simplicial complexes
and induce a series of natural surjective maps of simplicial profinite complexes.

\begin{theorem}\label{curve-complex}For $2g-2+n>0$, there is a series of natural 
$\cGG_{g,n}$-equivariant isomorphisms of simplicial profinite complexes:
$\ccC(S_{g,n})\sr{\sim}{\ra}L(\hP_{g,n})\sr{\sim}{\ra}L'(\hP_{g,n}).$
\end{theorem}

\begin{proof}
It is clearly enough to prove that the composition of the two above maps, which we call  $\Phi_\bt$,
is an isomorphism. Since, in a simplicial complex, simplices
are determined by their faces, it is enough to prove that $\Phi_0$ is injective.

In order to prove the claim, we will show that the actions of $\cGG_{g,n}$ on   
$\ccC(S_{g,n})_0$ and on $L'(\hP_{g,n})_0$ have same orbits and same stabilizers. 
Before we proceed further, let us give a geometric description of the natural faithful
representations, introduced in Section~\ref{completions}:
$$\check{\rho}_{g,n}\co\cGG_{g,n}\hookra\out(\hP_{g,n})\hspace{1cm}\mbox{and}\hspace{1cm}
\tilde{\rho}_{g,n}\co\cGG_{g,n+1}\hookra\aut(\hP_{g,n}).$$

We are going to interpret these representations as monodromy representations associated
to some profinite coverings of the universal curve over the moduli stack of stable curves
(for the general theory of these spaces, see \cite{V-W}).

Let $\tcM$ (respectively, $\tcC$) be the inverse limit of all representable geometric 
level structures over $\ccM_{g,n}$ (respectively, over $\ccM_{g,n+1}$).
Let also $\ol{\cC}:=\ccM_{g,n+1}\times_{\ccM_{g,n}}\tcM$ be the pull-back of the universal 
proper curve. They are connected by the commutative diagram:
$$\begin{array}{ccc}
\tcC&&\\
\da&\searrow&\\
\ol{\cC}&\ra&\!\!\!\!\!\!\tcM\\
\da&{\sst\square}&\!\!\!\!\!\!\da\\
\,\,\,\,\ccM_{g,n+1}\!\!\!\!\!&\ra&\ccM_{g,n}.
\end{array}$$

Let us observe that, by Definitions~\ref{finite complex} and \ref{geocompletion}, 
the nerve of the D--M boundary $\dd\tcM$ of the profinite covering $\tcM$ is 
described by the procongruence curve complex $\ccC(S_{g,n})$.

The procongruence Teichm\"uller group $\cGG_{g,n}$ identifies
with the deck transformation groups of the profinite Galois coverings $\tcM\ra\ccM_{g,n}$ 
and $\ol{\cC}\ra\ccM_{g,n+1}$, while $\cGG_{g,n+1}$ identifies with the deck 
transformation group of the profinite Galois covering $\tcC\ra\ccM_{g,n+1}$. 

Moreover, the kernel $\hP_{g,n}$ of the natural epimorphism 
$\cGG_{g,n+1}\tura\cGG_{g,n}$ consists of the deck transformations of the covering 
$\tcC\ra\ccM_{g,n+1}$ which preserve the 
fibers of the natural morphism $\tcC\ra\tcM$. In this way, the profinite group $\hP_{g,n}$ is 
canonically identified, modulo inner automorphisms, with the logarithmic 
algebraic fundamental groups of the fibers of the morphism $\ol{\cC}\ra\tcM$, with respect 
to the logarithmic structure associated to the divisor of marked and singular points. 

A good reference for the general theory of algebraic fundamental groups of log schemes
is Hoshi's paper \cite{H}. However, these groups, at least for a stable marked complex 
projective curve $C$, can be described in a more elementary fashion.

Let $\hat{C}$ be the real oriented blow-up of $C$ along the divisor of marked and singular points. 
Then, $\hat{C}$ is a compact orientable Riemann surface with boundary and the blow-up map 
$\hat{C}\ra C$ contracts a set of disjoint non-peripheral s.c.c.'s on $\hat{C}$ onto the nodes of $C$ 
and its boundary components onto the marked points of $C$.

Let $C^{log}$ be the curve $C$ endowed with the logarithmic structure associated to 
the divisor of marked and singular points.
A morphism $X^{log}\ra C^{log}$ of logarithmic schemes is log \'etale if and only if the underlying
morphism of schemes $X\ra C$ pulls back to a topological covering
map $X\times_C\hat{C}\ra\hat{C}$ via the blow-up map $\hat{C}\ra C$. 

The algebraic fundamental 
group $\hat{\pi}_1(C^{log})$ then classifies finite coverings of $C^{log}$ with the above property and
therefore it is naturally isomorphic to the profinite completion of the topological fundamental 
group of the Riemann surface $\hat{C}$.

The representation $\tilde{\rho}_{g,n}\co\cGG_{g,n+1}\ra\aut(\hP_{g,n})$ can be described 
as follows. For $x\in\tcM$, let $\tcC_x$ and $\ol{\cC}_x$ be, respectively,  the fibers of the 
morphisms $\tcC\ra\tcM$ and $\ol{\cC}\ra\tcM$. For all $x\in\tcM$, the profinite group 
$\hP_{g,n}$ identifies with the deck transformation group of the profinite covering 
$\tcC_x\ra\ol{\cC}_x$.
Let us denote by $\ol{f}$ the image of $f\in\cGG_{g,n+1}$ via the natural epimorphism 
$\cGG_{g,n+1}\ra\cGG_{g,n}$. The element $f\in\cGG_{g,n+1}$ induces, by restriction, an 
isomorphism of profinite curves $f\co\tcC_x\ra\tcC_{\ol{f}(x)}$. The automorphism 
$\tilde{\rho}_{g,n}(f)\co\hP_{g,n}\ra\hP_{g,n}$ is then defined, for $\alpha\in\hP_{g,n}$, by  
the assignment $\tilde{\rho}_{g,n}(f)(\alpha)=f^{-1}\alpha f$.

The profinite universal monodromy representation 
$\check{\rho}_{g,n}\co\cGG_{g,n}\ra\out(\hP_{g,n})$ is the outer representation associated 
to $\tilde{\rho}_{g,n}$. However, in the present geometric context, there is a more direct 
way to describe $\check{\rho}_{g,n}$. 

The profinite group $\cGG_{g,n}$ identifies with the deck 
transformation group of the profinite covering $\ol{\cC}\ra\ccM_{g,n+1}$. 
Let us denote by $\ol{\cC}_x^{log}$ and $\ol{\cC}^{log}_{f(x)}$ the fibers of the curve 
$\ol{\cC}\ra\tcM$ endowed with the logarithmic structure associated to the divisor of 
marked and singular points. An element $f\in\cGG_{g,n}$ determines by restriction an 
isomorphism of logarithmic algebraic curves $f\co\ol{\cC}_x^{log}\ra\ol{\cC}^{log}_{f(x)}$ 
and then of logarithmic algebraic 
fundamental groups $f_\natural\co\hat{\pi}_1(\ol{\cC}_x^{log})\ra\hat{\pi}_1(\ol{\cC}^{log}_{f(x)})$. 
These groups identify, modulo inner automorphisms, with the profinite group
$\hP_{g,n}$. Therefore, the isomorphism $f_\natural$ determines an element of 
$\out(\hP_{g,n})$.

The discrete universal monodromy representation $\rho\co\GG_{g,n}\ra\out(\Pi_{g,n})$
can be described in a similar way, where the Bers bordifications $\ol{T}_{g,n}$ and 
$\ol{T}_{g,n+1}$ and the universal Teichm\"uller curve $\ol{\mathfrak C}\ra\ol{T}_{g,n}$ 
play the same role of the profinite spaces $\tcM$, $\tcC$ and $\ol{\cC}$.

There are natural injective maps $\ol{T}_{g,n}\hookra\tcM$, $\ol{T}_{g,n+1}\hookra\tcC$ 
and $\ol{\mathfrak C}\hookra\ol{\cC}$. For $x\in\tcM$ in the image of $\ol{T}_{g,n}$,
the fiber $\ol{\cC}_x$ is contained in the image of $\ol{\mathfrak C}$ and the 
logarithmic topological fundamental group of $\ol{\cC}_x$ (here, by this, we just mean the 
topological fundamental group of the real oriented blow-up of $\ol{\cC}_x$ along its marked 
and singular points) identifies, modulo inner automorphisms, with $\Pi_{g,n}$. 

For a given s.c.c. $\gm\in\cL'$, 
let $(C_\gm,\psi)\in\dd T_{g,n}$ be such that $C_\gm$ is a stable $n$-pointed, genus $g$
curve with a single node and the degenerate marking 
$\psi\co S_{g,n}\ra C_\gm$ sends $\gm$ to the node. Let then $x\in\tcM$ be the 
image of the point $(C_\gm,\psi)$ via the natural embedding $\ol{T}_{g,n}\hookra\tcM$ and 
identify the stable curve $C_\gm$ with the fiber $\ol{\cC}_x$. 

Let $\tilde{C}_\gm$ be the inverse limit of all logarithmic \'etale coverings $\{C_\gm^\ld\}_{\ld\in\Ld}$ of 
$C_\gm$, with respect to the logarithmic structure associated to the divisor of marked and 
singular points. Let $\{R^\ld\subset C_\gm^\ld\}_{\ld\in\Ld}$ be an inverse system of irreducible 
components. Then, we call the inverse limit $\ilim_\ld\,R^\ld$ an {\it irreducible component of} 
$\tilde{C}_\gm$. Similarly, for $\{n^\ld\subset C_\gm^\ld\}_{\ld\in\Ld}$ an inverse system of nodes,  
we call $\ilim_\ld\,n^\ld$ a {\it node of} $\tilde{C}_\gm$.
For all $\ld$, the curve $C_\gm^\ld$ is semi-stable. In particular, each of its nodes is contained in 
exactly two irreducible components. Then, the same property holds for the
inverse limit $\tilde{C}_\gm$.

The curve $\tilde{C}_\gm$ identifies with the fiber $\tcC_x$ of the curve 
$\tcC\ra\tcM$. In this way, the group of deck transformations of the covering 
$\tilde{C}_\gm\ra C_\gm$ identifies with the profinite group $\hP_{g,n}$. 

From the above description, it is clear that two discrete s.c.c. $\gm,\gm'\in\cL'\subset\hL'$
are in the same $\cGG_{g,n}$-orbit if and only if the associated stable curves $C_\gm$ and 
$C_{\gm'}$ have the same topological type. This already implies that the natural 
continuous $\cGG_{g,n}$-equivariant map $\Phi_0\co\check{C}(S_{g,n})_0\ra L'(\hP_{g,n})_0$ 
induces a bijection between the $\cGG_{g,n}$-orbits of the two sets.

It remains to prove that $\Phi_0$ induces an isomorphism between the respective stabilizers.
A $\cGG_{g,n}$-orbit in $\ccC(S_{g,n})_0$ contains simplices in the image of the discrete complex
$C(S_{g,n})_0$. Hence, it is enough to show that, for a given s.c.c. $\gm$ on the Riemann surface 
$S_{g,n}$, the stabilizer $\cGG_\gm$ of its image in $\ccC(S_{g,n})_0$ equals the stabilizer 
$\cGG_{\Phi_0(\gm)}$ of its image in $\hL'$.

We are going to make use of the description of the simplicial profinite set $\check{C}(S_{g,n})_\bt$
in terms of framed, marked, graph decompositions of the homology of coverings of $S_{g,n}$
given in Section~7 of \cite{sym}. Let us briefly recall the construction given there, while, for more
details and precise definitions, we refer to Sections~6 and 7 of  \cite{sym}.

Let $K$ be a characteristic finite index subgroup of $\Pi_{g,n}$, with associated covering 
$p_K\co S_K\ra S_{g,n}$ and covering transformation group $G_K$. Let $n_K$ be the 
number of punctures on the Riemann surface $S_K$.

In Section~6 of  \cite{sym}, it is described the nerve of the D--M boundary
of the abelian level structure $\cM(S_K)^{(m)}$ over the moduli stack $\cM(S_K)$, which parametrizes 
complex curves homeomorphic to $S_K$. 

A $k$-simplex $\sg=\{\gm_e\}_{e\in E}$ of the curve complex $C(S_K)$ determines a stratum 
$\delta_\sg^{(m)}$ of codimension $k+1$ in the D--M boundary of $\ccM(S_K)^{(m)}$. 
The s.c.c.'s in $\sg$ determine a partition $\amalg_{v\in V} S_v:=S_K\ssm\sg$ of the Riemann surface 
$S_K$ in open connected subsurfaces $S_v$ of negative Euler characteristic. 

In its turn, this partition determines what we call a framed, $n_K$-marked, graph decompositions of 
the homology group $H_1(\ol{S}_K,\Z/m)$ (cf. Definitions~6.1 and 6.5 \cite{sym}). This is the graph 
$(D,Y_\sg)$ of abelian subgroups of $H_1(\ol{S}_K,\Z/m)$ whose vertex groups $D_v$, for $v\in V$,  
consist of the images of the natural homomorphisms $H_1(S_v,\Z/2)\ra H_1(\ol{S}_K,\Z/m)$ and 
whose edge groups $D_e$, for $e\in E$, are the cyclic subgroups of $H_1(\ol{S}_K,\Z/m)$ determined 
by the s.c.c.'s $\gm_e\in\sg$. Let us observe that the graph $Y_\sg$ is just the dual graph of the stable 
degeneration of the Riemann surface $S_K$ whose vanishing cycles are the s.c.c.'s in $\sg$.

Let $\cP_{n_K}$ be the set of subsets of the index set $\{1,\ldots,n_K\}$.
The marking $d\co V\ra\cP_{n_K}$ is then defined assigning to a vertex $v$ of the graph
$Y_\sg$ the indices of the punctures lying on the corresponding subsurface $S_v$ of $S_K$.
The frame epimorphism $\mu_e\co\Z\ra D_e$, for an edge $e\in E$ of the graph $Y_\sg$, is 
the natural epimorphism between the two cyclic subgroups associated to the 
s.c.c. $\gm_e$ in the homology groups $H_1(\ol{S}_K,\Z)$ and $H_1(\ol{S}_K,\Z/m)$, respectively.

By Theorem~6.7 \cite{sym}, the framed, $n_K$-marked, graph decomposition
$(D,Y_\sg,d,\{\mu_e\})$, modulo a natural equivalence relation, determines the stratum 
$\delta_\sg^{(m)}$, for any $m\geq 2$.

A simplex $\sg\in C(S_{g,n})$, via the covering $p_K\co S_K\ra S_{g,n}$, determines the simplex 
$\sg':=p_K^{-1}(\sg)$ in the curve complex $C(S_K)$. 
By the above construction, we then associate to $\sg$ the equivalence class
$[D,Y_{\sg'},d,\{\mu_e\}]$ of the framed, $n_K$-marked, graph decomposition of $H_1(\ol{S}_K,\Z/m)$
determined by the simplex $\sg'$. This decomposition has the further property of admitting a natural 
action of the covering transformation group $G_K$. 

In Section~7 \cite{sym}, for $m\geq 2$, it is defined the simplicial finite set
$\cG^{n_K}_{G_K}(H_1(\ol{S}_K,\Z/m))_\bt$ of framed, $n_K$-marked, graph decompositions of the 
homology $H_1(\ol{S}_K,\Z/m)$, endowed with a natural $G_K$-action. As observed in iv.),
Remark~7.6 \cite{sym}, if the subgroup $K$ satisfies
the hypotheses of Theorem~3.9 \cite{sym}, the natural map, defined above,
$$\mathfrak{g}_{K,(m)}\co C(S_{g,n})_\bt\ra\cG^{n_K}_{G_K}(H_1(\ol{S}_K,\Z/m))_\bt$$
is $\GG_{g,n}$-equivariant and factors through a continuous $\cGG_{g,n}$-equivariant map:
$$\check{\mathfrak{g}}_{K,(m)}\co\check{C}(S_{g,n})_\bt\ra
\cG^{n_K}_{G_K}(H_1(\ol{S}_K,\Z/m))_\bt.$$

By Corollary~7.8 \cite{sym}, if $\{K\}$ is a cofinal system of finite index subgroups of $\Pi_{g,n}$ 
satisfying the hypotheses of Theorem~3.9 \cite{sym} and $m$ an integer $\geq 2$, there is
a natural continuous, injective, $\cGG_{g,n}$-equivariant map of simplicial profinite sets:
$$\prod_{K}\check{\mathfrak{g}}_{K,(m)}\co\check{C}(S_{g,n})_\bt\hookra
\prod_{K}\cG^{n_K}_{G_K}(H_1(\ol{S}_K,\Z/m))_\bt.$$
The existence of a cofinal system with these properties is then ensured by Lemma~3.10 \cite{sym}.

As observed in Remark~6.8 \cite{sym}, with $\Z/2$-coefficients, all framings of the edge groups 
are equivalent, i.e. a simplex of $\cG^{n_K}_{G_K}(H_1(\ol{S}_K,\Z/2))_\bt$ is already determined by 
the equivalence class of the corresponding $n_K$-marked, graph decomposition. 


So, let us describe the map $\check{\mathfrak{g}}_{K,(2)}$ for
a given s.c.c. $\gm$ on the Riemann surface $S_{g,n}$. The s.c.c. $\gm$ determines a partition 
$\amalg_{v\in V} S_v:=S_K\ssm p_K^{-1}(\gm)$ of the Riemann surface $S_K$ and this determines,
as above, an $n_K$-marked, graph $(D^K,Y^K,d^K)$ of abelian subgroups of 
$H_1(\ol{S}_K,\Z/2)$. This comes with a natural $G_K$-action, 
induced by the deck transformations of the covering map $p_K\co S_K\ra S_{g,n}$.

The map $\check{\mathfrak{g}}_{K,(2)}$ then associates to the $0$-simplex of 
$\check{C}(S_{g,n})_\bt$ associated to the s.c.c. $\gm$ on $S_{g,n}$
the corresponding equivalence class $[D^K,Y^K,d^K]$ of $n_K$-marked, graph 
decomposition of the homology group $H_1(\ol{S}_K,\Z/2)$, endowed
with a natural $G_K$-action. 

By Lemma~3.10 \cite{sym}, we can assume that all elements in the inverse system $\{K\}$ satisfy the 
hypotheses of Theorem~3.9 \cite{sym}, i.e. they are such that the inverse image 
$p_K^{-1}(\gm)$ contains no separating curves or cut pairs. So that the cyclic subgroups of the 
homology group $H_1(\ol{S}_K,\Z/2)$ determined by the s.c.c.'s contained in $p_K^{-1}(\gm)$ 
are all distinct and non-trivial. 

In order to prove that $\Phi_0$ induces an isomorphism between 
the respective stabilizers of $0$-simplices, we then have to show 
that the stabilizer $\cGG_{\Phi_0(\gm)}$ preserves the equivalence class $[D^K,Y^K,d^K]$, 
associated to the s.c.c. $\gm$, for a cofinal system of characteristic finite index subgroups $\{K\}$ 
of $\Pi_{g,n}$ satisfying the hypotheses of Theorem~3.9 \cite{sym}. 

Let us consider again the covering $\tilde{C}_\gm\ra C_\gm$ whose deck transformation group
was identified with the profinite group $\hP_{g,n}$ via the degenerate marking 
$\psi\co S_{g,n}\ra C_\gm$.

Let $\gm_\ast$ be an element of the fundamental group $\Pi_{g,n}$ which generates a cyclic 
subgroup in the equivalence class of the given $\gm\in\cL'$.
Elements of $\hP_{g,n}$ in the conjugacy class of $\gm_\ast$ are in natural bijective 
correspondence with the set $\cN_\gm$ of nodes of $\tilde{C}_\gm$. Let then 
$\{\gm_s\}_{s\in\cN_\gm}$ be the set of conjugates of $\gm_\ast$ in the profinite group $\hP_{g,n}$.

The stabilizer of the node $s\in\cN_\gm$ for the action by $\hP_{g,n}$ is the closed cyclic group $I_s$ 
generated by $\gm_s$. Moreover, the action of $I_s$ on the curve $\tilde{C}_\gm$ has
only the node $s$ as fixed point. Therefore, the subgroups $\{I_s\}_{s\in c_\gm}$ are their own 
normalizers in $\hP_{g,n}$. The same argument clearly applies to the stabilizers of marked points 
of $\tilde{C}_\gm$. In particular, it holds:

\begin{lemma}\label{central}Let $1\neq x\in\Pi_{g,n}$, for $2g-2+n>0$, be representable by
a s.c.c. on $S_{g,n}$. Then, the closed subgroup generated by $x^m$ in 
$\hat{\Pi}_{g,n}$, for $m\in\ZZ\ssm\{0\}$, has for normalizer the closed subgroup 
generated by $x$. Moreover, for $y\in\hP_{g,n}\ssm\ol{\langle x\rangle}$, it holds 
 $\ol{\langle x\rangle}\cap y\ol{\langle x\rangle}y^{-1}=\{1\}$.
\end{lemma}

Let us now assume that the s.c.c. $\gm$ is non-separating, the other case can be treated similarly.
Let $\dot{C}_\gm$ be the smooth curve obtained removing the singular point from $C_\gm$. 
Then, $\dot{C}_\gm$ has just one connected component. For suitable choices of base points, the 
fundamental group of the subsurface $\psi^{-1}(\dot{C}_\gm)$ of $S_{g,n}$ can be identified with a 
subgroup $\Pi_\gm$ of the fundamental group $\Pi_{g,n}$ which contains $\gm_\ast$.

The profinite completion $\hP_\gm$ of $\Pi_\gm$ is identified with 
the deck transformation group of the restriction of the covering $\tilde{C}_\gm\ra C_\gm$ 
to one of the two irreducible components of $\tilde{C}_\gm$ which contain the node $n_\ast$
corresponding to $\gm_\ast$. There is then a bijective correspondence between 
irreducible components of $\tilde{C}_\gm$ and conjugates of the subgroup $\hP_\gm$ 
in $\hP_{g,n}$. In particular, the subgroup $\hP_\gm$ is its own normalizer in $\hP_{g,n}$. 
Let us denote by $\hP_\gm'$ the conjugate of $\hP_\gm$ corresponding 
to the other irreducible component of $\tilde{C}_\gm$ which also contains the node $n_\ast$.

Let $f\in\cGG_{g,n}$ be an element which preserves the conjugacy class of the closed 
subgroup $\ol{\langle\gm_\ast\rangle}$ of $\hP_{g,n}$ generated by $\gm_\ast$ and let 
$\tilde{f}\in\aut(\hP_{g,n})$ be a lift of $f$ such that it holds
$\tilde{f}(\ol{\langle\gm_\ast\rangle})=\ol{\langle\gm_\ast\rangle}$. The isomorphism between  
fibers $\tilde{f}\co\tcC_x\ra\tcC_{\tilde{f}(x)}$, induced by $\tilde{f}$, then sends the irreducible 
component of $\tcC_x$ stabilized by the subgroup $\hP_\gm$ of $\hP_{g,n}$ to 
the irreducible component of $\tcC_{\tilde{f}(x)}$ stabilized by either $\hP_\gm$ or $\hP_\gm'$.
This means that $\tilde{f}$ either preserves both subgroups $\hP_\gm$ and $\hP_\gm'$ or 
swaps them.

Let $\hat{K}$ be the profinite completion of a subgroup $K$ of $\Pi_{g,n}$ satisfying the hypotheses 
of Theorem~3.9 \cite{sym}. The conjugacy class of the closed subgroup 
$\hat{K}\cap\ol{\langle\gm_\ast\rangle}$ of $\hP_{g,n}$ is 
mapped by the series of natural epimorphisms 
$H_1(\hat{K})\cong H_1(S_K)\tura H_1(\ol{S}_K,\Z/2)$ onto the $G_K$-orbit of the edge groups
of $(D^K,Y^K,d^K)$, while the conjugacy class of the closed subgroups $\hat{K}\cap\hP_\gm$ 
and $\hat{K}\cap\hP_\gm'$ is mapped onto the $G_K$-orbit of the vertex groups of $(D^K,Y^K,d^K)$. 

It follows that the given $f\in\cGG_{\Phi_0(\gm)}$ preserves vertex and edge groups of the 
graph of groups $(D^K,Y^K)$. Under the current hypotheses, the equivalence class $[D^K,Y^K]$ 
of $(D^K,Y^K)$ is determined by its vertex and edge groups. Therefore, the action of $f$ 
preserves at least the set of simplices of $\cG^{n_K}_{G_K}(H_1(\ol{S}_K,\Z/2))_\bt$ whose 
underlying graph decompositions are equivalent to the graph of groups $(D^K,Y^K)$. 

On the other hand, the $G_K$-orbit of markings for the vertices of $Y^K$ is obviously preserved. 
Thus, in conclusion, the equivalence class $[D^K,Y^K,d^K]$ of $(D^K,Y^K,d^K)$ is preserved by the 
action of $f$. From Corollary~7.8 in \cite{sym}, it then follows that $f$ preserves the $0$-simplex of 
$\ccC(S_{g,n})_0$ determined by $\gm$. This concludes the proof of Theorem~\ref{curve-complex}.

\end{proof}

Thanks to Theorem~\ref{curve-complex}, the procongruence curve complex $\ccC(S_{g,n})$ 
is realized, more concretely, in terms of profinite s.c.c.'s on $S_{g,n}$. From now on, we will 
identify the procongruence curve complex $\ccC(S_{g,n})$ with the complex of profinite 
curves $L(\hP_{g,n})$.

The orbits of the complex of profinite curves $L(\hP_{g,n})$ for the action of the procongruence 
Teichm\"uller group $\cGG_{g,n}$ are the closures of the images of the orbits of the discrete curve 
complex $C(S_{g,n})$ for the action of the Teichm\"uller group $\GG_{g,n}$ and, by the proof of
Theorem~\ref{curve-complex}, they are all distinct. The following notion is then well defined:

\begin{definition}\label{topological type}For $2g-2+n>0$, let $\sg\in L(\hP_{g,n})$. {\it The 
topological type} of $\sg$ is the homeomorphism type of the Riemann surface $S_{g,n}\ssm\sg'$,
where $\sg'\in C(S_{g,n})$ is such that its image in $L(\hP_{g,n})$ is in the
$\cGG_{g,n}$-orbit of $\sg$.
\end{definition}

Let us now describe the stabilizers for the action of $\cGG_{g,n}$ on the complex of profinite curves 
$L(\hP_{g,n})$. Again, it is enough to describe the
stabilizers of simplices in the image of the discrete curve complex.
Given a simplex $\sg\in C(S_{g,n})$, let us denote also by $\sg$ its image in $L(\hP_{g,n})$ 
and by $\cGG_\sg$ the corresponding $\cGG_{g,n}$-stabilizer. 

\begin{theorem}\label{stab pro-curves}Let $\sg\in L(\hP_{g,n})$, for $2g-2+n>0$, be the image 
of a simplex of $C(S_{g,n})$ determined by the set $\{\gm_0,\ldots,\gm_k\}$ of s.c.c. on the 
Riemann surface $S_{g,n}$. Suppose that 
$$S_{g,n}\ssm\{\gm_0,\ldots,\gm_k\}\cong S_{g_1,n_1}\amalg\ldots\amalg S_{g_h,n_h}$$ 
and let $\Sigma_{\sg^{\pm}}$ be the group of permutations on the set of oriented s.c.c.'s
$\{\vec{\gm}_0^\pm,\ldots,\vec{\gm}_k^\pm\}$.
Then, the stabilizer $\cGG_\sg$ of $\sg$ in $L(\hP_{g,n})$ for the action of
$\cGG_{g,n}$ fits into the two exact sequences:
$$\begin{array}{c}
1\ra\cGG_{\vec{\sg}}\ra\cGG_\sg\ra\Sigma_{\sg^{\pm}},\\
1\ra\bigoplus\limits_{i=0}^k\ZZ\cdot\tau_{\gm_i}\ra\cGG_{\vec{\sg}}
\ra\cGG_{g_1,n_1}\times\dots\times\cGG_{g_h,n_h}\ra 1.
\end{array}$$
\end{theorem}

By Proposition~\ref{stab=clos} and Theorem~\ref{curve-complex}, for a simplex $\sg\in C(S_{g,n})$, 
the stabilizer $\cGG_\sg$ of its image in $L(\hP_{g,n})$ for the action of $\cGG_{g,n}$ is the closure 
$\ol{\GG}_\sg$ in the profinite group $\cGG_{g,n}$ of the discrete stabilizer $\GG_\sg$. 

By $ii.)$ Theorem~\ref{comparison}, injectivety of the second map from the left in the bottom 
sequence of Theorem~\ref{stab pro-curves} is equivalent to the claim that Looijenga levels induce 
the profinite topology on the abelian subgroup $\bigoplus_{i=0}^k\Z\cdot\tau_{\gm_i}$ of $\GG_{g,n}$. 

This claim then follows from the explicit description of the local monodromy 
representation for Looijenga level structures, for $K$ satisfying the hypotheses of 
Lemma~ 3.10 \cite{sym} and all $\ell, m\geq 2$ (cf. Proposition~3.11 \cite{sym}), which implies:
$$\GG^{K_\ell,(m)}\cap\bigoplus\limits_{i=0}^k\Z\cdot\tau_{\gm_i}
\leq\bigoplus\limits_{i=0}^k m\,\Z\cdot\tau_{\gm_i}.$$

By Proposition~6.6 \cite{PFT}, there is a natural representation:
$$\check{\rho}_\sg\co\cGG_{\vec{\sg}}\ra\cGG_{g_1,n_1}\times\dots\times\cGG_{g_h,n_h}.$$
Thus, right exactness of the bottom sequence of Theorem~\ref{stab pro-curves} follows if we 
show that it holds $\ker\,\check{\rho}_\sg=\bigoplus_{i=0}^k\ZZ\cdot\tau_{\gm_i}$.

Proceeding by induction on the rank of the simplex, it is not difficult to reduce to the case of the 
$0$-simplex determined by a s.c.c. $\gm$ on $S_{g,n}$. So, we are reduced to prove the
following lemma:

\begin{lemma}\label{closure}Let $\gm$ be a s.c.c. on $S_{g,n}$ and let $\cGG_{\vec{\gm}}$ 
be the subgroup of $\cGG_{g,n}$ consisting of elements which fix the s.c.c. $\gm$ and an 
orientation on it. It holds:
\begin{enumerate}
\item If $\gm$ is non-separating, the profinite monodromy representation $\check{\rho}_{g,n}$ 
induces a representation $\check{\rho}_\gm\co\cGG_{\vec{\gm}}\ra\out(\hP_{g-1,n+2})$, whose
kernel is topologically generated by the Dehn twist $\tau_\gm$.
\item If $\gm$ is separating and $S_{g,n}\ssm\gm\cong S_{g_1,n_1+1}\amalg S_{g_2,n_2+1}$,
the profinite monodromy representation $\check{\rho}_{g,n}$ induces a representation 
$\check{\rho}_\gm\co\cGG_{\vec{\gm}}\ra\out(\hP_{g_1,n_1+1})\times\out(\hP_{g_2,n_2+1})$,
whose kernel is topologically generated by the Dehn twist $\tau_\gm$.
\end{enumerate}
\end{lemma}

\begin{proof}Let us assume that $\gm$ is a non-separating s.c.c. on $S_{g,n}$. The case of a separating curve can be treated similarly.
Let us give to the group $\Pi_{g,n}$ the standard presentation:
$$\Pi_{g,n}=\langle\alpha_1, \dots \alpha_g,\beta_1,\dots,\beta_g, u_1,\dots,u_n|\;
\prod_{i=1}^g[\alpha_i,\beta_i] \cdot u_n\cdots u_1\rangle,$$ 
where $u_i$, for $i=1,\dots,n$, is a simple loop around the  puncture $P_i$. 

Let us assume that $\beta_g$ is freely isotopic to $\gm$ and put 
$u_{n+1}:=\beta_g^{-1}$ and $u_{n+2}:=\alpha_g\beta_g\alpha_g^{-1}$. For a suitable choice of 
base point, the fundamental group of $S_{g,n}\ssm\gm$ is then identified with the subgroup 
$\Pi_{g-1,n+2}$ of $\Pi_{g,n}$ of presentation:
$$ \Pi_{g-1,n+2}=\langle\alpha_1, \dots \alpha_{g-1},\beta_1,\dots,\beta_{g-1}, 
u_1,\dots,u_{n+2}|\;\prod_{i=1}^{g-1}[\alpha_i,\beta_i] \cdot u_{n+2}\cdots u_1\rangle.$$
 
Let us maintain the notations introduced in the proof of Theorem~\ref{curve-complex}.
Elements of $\hP_{g,n}$ in the conjugacy class of $\beta_g$ are then in natural bijective 
correspondence with the set $\cN_\gm$ of nodes of $\tilde{C}_\gm$. 


For a given $f\in\cGG_{\vec{\gm}}<\cGG_{g,n}$, there is a lift $\tilde{f}\in\aut(\hP_{g,n})$ 
such that $\tilde{f}(\beta_g)=\beta_g$. From Lemma~\ref{central}, it follows that such lift $\tilde{f}$ 
is uniquely determined by $f$, modulo inner automorphisms by powers of $\beta_g$.

By the remarks made in the proof of Theorem~\ref{curve-complex}, it is clear that the isomorphism 
between fibers $\tilde{f}\co\tcC_x\ra\tcC_{\tilde{f}(x)}$, induced by $\tilde{f}$, sends the irreducible 
component of $\tcC_x$ stabilized by the subgroup $\hP_{g-1,n+2}$ of $\hP_{g,n}$ to 
the irreducible component of $\tcC_{\tilde{f}(x)}$ stabilized by 
the same subgroup and hence induces an automorphism 
$\tilde{f}'\co\hP_{g-1,n+2}\sr{\sim}{\ra}\hP_{g-1,n+2}$. We
let then $\check{\rho}_\gm(f)$ be the image of $\tilde{f}'$ in $\out(\hP_{g-1,n+2})$.

In order to complete the proof of the lemma, it is now enough to show that the kernel of 
$\check{\rho}_\gm$ is generated by the Dehn twist $\tau_\gm$. As above, we just treat the 
case of a non-separating s.c.c. $\gm$ on $S_{g,n}$.

Let us keep the above notations and let moreover $t:=\alpha_g$. We then get the standard 
presentation of $\Pi_{g,n}$ as an HNN extension of its free subgroup $\Pi_{g-1,n+2}$:
$$\Pi_{g,n}=\langle t,\Pi_{g-1,n+2}|\; t u_{n+1}^{-1}t^{-1}=u_{n+2}\rangle.$$

In terms of graphs of groups, we are saying that $\Pi_{g,n}$ is naturally isomorphic to the 
fundamental group of the loop of groups having for vertex group $\Pi_{g-1,n+2}$ and for edge 
group the free cyclic group spanned by $t$. 


For a given $f\in\cGG_\gm$, we have seen that it admits a lift $\tilde{f}$ to an automorphism of 
$\hP_{g,n}$ which preserves the closed subgroup $\hP_{g-1,n+2}$ and $\check{\rho}_\gm(f)$ 
is defined to be the outer automorphism induced by $\tilde{f}$ on this subgroup.

Let $f\in\cGG_\gm$ be such that $\check{\rho}_\gm(f)=1$. It then admits a lift 
$\tilde{f}\in\aut(\hP_{g,n})$ which 
restricts to the identity on $\hP_{g-1,n+2}$. In particular, it holds $\tilde{f}(u_{n+1})=u_{n+1}$ 
and $\tilde{f}(u_{n+2})=u_{n+2}$. The action of $\tilde{f}$ hence satisfies the condition:
$$\tilde{f}(u_{n+2})=\tilde{f}(t)u_{n+1}^{-1}\tilde{f}(t)^{-1}=u_{n+2}=t u_{n+1}^{-1}t^{-1}.$$ 
By Lemma~\ref{central}, this identity implies $\tilde{f}(t)=tu_{n+1}^k$, for some $k\in\ZZ$. 
It then follows that $f$ equals the power $\tau_\gm^{-k}\in\cGG_{g,n}$, since the 
power $\tau_\gm^{-k}$ also lifts to an automorphism $\tilde{\tau}_\gm^{-k}$ of $\hP_{g,n}$ 
such that $\tilde{\tau}_\gm^{-k}(t)=t u_{n+1}^{k}$ and which
restricts to the identity on the subgroup $\hP_{g-1,n+2}$.

\end{proof}

\begin{remark}\label{link}Theorem~\ref{curve-complex} and Theorem~\ref{stab pro-curves} yield, 
in particular, a description of the links of simplices in the complex of profinite curves 
$L(\hP_{g,n})$ analogous to the description of links of simplices in the curve complex $C(S_{g,n})$. 
Indeed, with the same notations of Theorem~\ref{stab pro-curves}, the link 
$\mathrm{Lk}(\sg)$ of the simplex $\sg$ in the simplicial profinite complex $L(\hP_{g,n})$ is 
naturally isomorphic to the join of simplicial profinite complexes:
$$L(\hP_{g_1,n_1})\ast\ldots\ast L(\hP_{g_h,n_h}).$$
\end{remark}

\section{Profinite Dehn twists in $\cGG_{g,n}$}\label{profinite twists}
A basic feature of classical Teichm\"uller theory (cf. \S 2, \cite{Birman}) is the fact that 
the set $\cL$ of isotopy classes of non-peripheral s.c.c.'s on $S_{g,n}$ parametrizes the set of Dehn 
twists of $\GG_{g,n}$, which is the standard set of generators for this group. In other words, the 
assignment $\gm\mapsto\tau_\gm$, for $\gm\in\cL$, defines an embedding 
$d\co\cL\hookra\GG_{g,n}$, for $2g-2+n>0$.

The set $\{\tau_\gm\}_{\gm\in\cL}$ of all Dehn twists of $\GG_{g,n}$ is closed
under conjugation and falls in a finite set of conjugacy classes which are in bijective
correspondence with the possible topological types of the Riemann surface $S_{g,n}\ssm\gm$.
So, it is natural to define, for a given profinite completion $\GG_{g,n}'$ of the
Teichm\"uller group $\GG_{g,n}$, the set of {\it profinite Dehn twists} of $\GG'_{g,n}$
to be the closure of the image of the set $\{\tau_\gm\}_{\gm\in\cL}$ inside $\GG'_{g,n}$,
which is the same as the union of the conjugacy classes in $\GG'_{g,n}$ of the images 
of the Dehn twists of $\GG_{g,n}$. 

More generally, there is a natural $\GG_{g,n}$-equivariant map $d_k\co\cL\ra\cGG_{g,n}$, defined 
by the assignment $\gm\mapsto\tau_\gm^k$, where $\GG_{g,n}$ acts by conjugation on 
$\cGG_{g,n}$. From the universal property of the $\cGG_{g,n}$-completion and 
Theorem~\ref{curve-complex}, it then follows that the map $d_k$ extends to a continuous 
$\cGG_{g,n}$-equivariant map $\hat{d}_k\co\hat{\cL}\ra\cGG_{g,n}$, whose image is the set of
$k$-th powers of profinite Dehn twists. 

In particular, a profinite s.c.c. $\gm\in\hat{\cL}$ determines a profinite Dehn twist, 
which we denote by $\tau_\gm$, in the procongruence Teichm\"uller group $\cGG_{g,n}$.
The following theorem says that this provides a parametrization of the set 
of profinite Dehn twists of $\cGG_{g,n}$:

\begin{theorem}\label{parametrize}For $2g-2+n>0$ and any
$k\in\ZZ\ssm{0}$, there is a natural injective map 
$\hat{d}_k\co\hat{\cL}\hookra\cGG_{g,n}$ which assigns to a profinite s.c.c. $\gm\in\hat{\cL}$
the $k$-th power of the profinite Dehn twist  $\tau_\gm$. 
\end{theorem}

In order to prove Theorem~\ref{parametrize}, it is necessary a long preliminary digression. 

 \begin{definition}Let us assume that the base point of the fundamental group $\Pi_{g,n}$ of 
$S_{g,n}$ is the marked point $P_{n+1}$. Let $\cL^\ast_{g,n}$ be the set of relative isotopy 
classes of $P_{n+1}$-pointed oriented non-peripheral s.c.c.'s on $S_{g,n}$. 
This set embeds in the fundamental group $\Pi_{g,n}$ (cf. Theorem~3.4.15 in \cite{C-G-K-Z}).

Let us then define the set of {\it $P_{n+1}$-pointed profinite oriented non-peripheral s.c.c.'s} 
$\hat{\cL}^\ast_{g,n}$ to be the closure of $\cL^\ast_{g,n}$ inside the profinite 
completion $\hP_{g,n}$. The quotients of the sets $\cL^\ast_{g,n}$ and $\hat{\cL}^\ast_{g,n}$ by 
the action of inner automorphisms are, respectively, the set of {\it oriented non-peripheral s.c.c.'s}
$\cL^{or}_{g,n}$ and of {\it profinite oriented non-peripheral s.c.c.'s} $\hL^{or}_{g,n}$.
\end{definition}

\begin{remark}\label{forgetting}There is a natural continuous surjective map 
$\hat{q}\co\hat{\cL}^\ast_{g,n}\tura\hat{\cL}_{g,n}$. We then let $\gm:=\hat{q}(\vec{\gm})$ and say 
that $\vec{\gm}\in\hat{\cL}^\ast_{g,n}$ lifts the profinite s.c.c. $\gm\in\hat{\cL}_{g,n}$.
\end{remark}

For $2g-2+n>0$, there is a natural monomorphism $i\co\Pi_{g,n}\hookra\GG_{g,n+1}$ which, 
for the relative isotopy class of a $P_{n+1}$-pointed oriented s.c.c. $\vec{\gm}$, is described as 
follows. Let $U_\gm$ be a closed tubular neighborhood of $\vec{\gm}$. The boundary of $U_\gm$ 
consists of two s.c.c.'s on $S_{g,n+1}$. Let us denote by $\gm_1$ the boundary component of 
$U_\gm$ which lies at the right of $\vec{\gm}$, with respect to its orientation, and by $\gm_2$ 
the one which lies at its left. The isotopy classes of the pair of s.c.c.'s $\{\gm_1,\gm_2\}$ 
only depend on the class $[\vec{\gm}]$ of $\vec{\gm}$ in $\Pi_{g,n}$.
The monomorphism $i$ then assigns to the class of $\vec{\gm}$ the product of Dehn twists 
$\tau_{\gm_1}\cdot\tau_{\gm_2}^{-1}$ (see \S 3, \cite{Birman}, for more details on this construction). 

It is easy to see that the assignment $[\vec{\gm}]\mapsto[\gm_i]$ is $\GG_{g,n+1}$-equivariant,
for $i=1,2$, where $\GG_{g,n+1}$ acts on its subgroup $\Pi_{g,n}$ by conjugation.
In particular, this defines two $\GG_{g,n+1}$-equivariant maps 
$s_i\co\cL^\ast_{g,n}\ra\cL_{g,n+1}$, for $i=1,2$, such that the product map 
$s_1\times s_2\co\cL^\ast_{g,n}\ra\cL_{g,n+1}\times\cL_{g,n+1}$ is injective.

It is easy to check that the image of both maps $s_1$ and $s_2$ is the subset $\cL_{g,n+1}^b$ of 
$\cL_{g,n+1}$ consisting of non-peripheral s.c.c.'s which do not bound a two-punctured disc, with one 
of the two punctures labeled by $P_{n+1}$. Let us denote by $\hL_{g,n+1}^b$ the closure 
of $\cL_{g,n+1}^b$ inside the profinite set $\hL_{g,n+1}$. The following is basically a reformulation 
of Proposition~2.7 \cite{Hyp2}:
 
\begin{proposition}\label{extend}For $2g-2+n>0$, the profinite set $\hat{\cL}^\ast_{g,n}$ is the 
$\cGG_{g,n+1}$-completion of $\cL^\ast_{g,n}$. In particular, the maps $s_i$ extend 
naturally to $\cGG_{g,n+1}$-equivariant surjective continuous maps 
$\hat{s}_i\co\hL^\ast_{g,n}\tura\hL_{g,n+1}^b$, for $i=1,2$, such that the product map
$\hat{s}_1\times\hat{s}_2\co\hL^\ast_{g,n}\ra\hL_{g,n+1}^b\times\hL_{g,n+1}^b$ is injective.
\end{proposition}

\begin{proof}In order to prove that $\hat{\cL}^\ast_{g,n}$ is the $\cGG_{g,n+1}$-completion
$\check{\cL}^\ast_{g,n}$ of $\cL^\ast_{g,n}$, let us observe that, for $2g-2+n>0$, there is a 
$\GG_{g,n}$-equivariant isomorphism:
$\cL^{or}_{g,n}\cong\cL_{g,n}\times\{\pm 1\}.$

By Theorem~\ref{curve-complex}, $\hL_{g,n}$ is the $\cGG_{g,n}$-completion
of $\cL_{g,n}$. Hence, the above isomorphism induces an isomorphism of 
$\cGG_{g,n}$-completions
$\check{\cL}^{or}_{g,n}\cong \hL_{g,n}\times\{\pm 1\}$
and so $\check{\cL}^{or}_{g,n}\cong\hL^{or}_{g,n}$.

But now, by definition, $\cL^{or}_{g,n}=\cL^\ast_{g,n}/\Pi_{g,n}$. Hence, it holds
$\check{\cL}^{or}_{g,n}\cong\check{\cL}^\ast_{g,n}/\hP_{g,n}$ as well. On the other hand, 
also by definition, it holds $\hL^{or}_{g,n}=\hL^\ast_{g,n}/\hP_{g,n}$.
Therefore, it holds $\hL^\ast_{g,n}/\hP_{g,n}\cong\check{\cL}^\ast_{g,n}/\hP_{g,n}$, 
which implies that $\hat{\cL}^\ast_{g,n}$ is the $\cGG_{g,n+1}$-completion of $\cL^\ast_{g,n}$.

By the universal property of the $\cGG_{g,n+1}$-completion, the maps $s_i$ then extend 
to surjective continuous $\cGG_{g,n+1}$-equivariant maps 
$\hat{s}_i\co\hL^\ast_{g,n}\tura\hL_{g,n+1}^b$, for $i=1,2$. 

The product map 
$\hat{s}_1\times\hat{s}_2$ is injective, since it factors the injective map which assign to 
$\vec{\gm}\in\hat{\cL}^\ast_{g,n}$ the product of profinite Dehn twists
$\tau_{\hat{s}_1(\vec{\gm})}\tau_{\hat{s}_2(\vec{\gm})}\in\cGG_{g,n+1}$.

\end{proof}

\begin{remark}\label{bpm}For $2g-2+n>0$ and $k\in\ZZ\ssm{0}$, the assignment 
$\vec{\gm}\mapsto\tau^{k}_{\hat{s}_1(\vec{\gm})}\tau^{-k}_{\hat{s}_2(\vec{\gm})}$, 
for $\vec{\gm}\in\hL^\ast_{g,n}$, defines a 
$\cGG_{g,n+1}$-equivariant continuous injective map $\hL^\ast_{g,n}\hookra\cGG_{g,n+1}$.

The injectivety of the map $\hat{s}_1\times\hat{s}_2$ then implies that the centralizer in 
$\cGG_{g,n+1}$ of the product of powers of Dehn twists 
$\tau_{\hat{s}_1(\vec{\gm})}^{k}\tau_{\hat{s}_2(\vec{\gm})}^{-k}$
is the stabilizer $\cGG_\sg$ of the $1$-simplex 
$\sg:=\{\hat{s}_1(\vec{\gm}),\hat{s}_2(\vec{\gm})\}\in L(\hP_{g,n+1})$.
\end{remark}

In order to prove Theorem~\ref{parametrize}, it is necessary to describe the $\hP_{g,n}$-conjugacy
classes of the powers of profinite Dehn twists in $\cGG_{g,n+1}$. Let 
$\check{p}\co\cGG_{g,n+1}\ra\cGG_{g,n}$ be the epimorphism induced by filling in the $(n+1)$-th 
puncture on $S_{g,n+1}$, whose kernel is identified with $\hP_{g,n}$.

\begin{theorem}\label{conjugacy classes}For $2g-2+n>0$, let be given $\gm\in\hL_{g,n}$ and 
$k\in\ZZ\ssm\{0\}$. Let then $\vec{\gm}\in\hL^\ast_{g,n}$ be a pointed oriented profinite s.c.c.
which lifts $\gm$. The set of powers of Dehn twists in $\cGG_{g,n+1}$ which lift
$\tau_\gm^k\in\cGG_{g,n}$, along the epimorphism $\check{p}\co\cGG_{g,n+1}\ra\cGG_{g,n}$, 
is the union of the $\hP_{g,n}$-conjugacy orbits of 
$\tau_{\hat{s}_1(\vec{\gm})}^k$ and $\tau_{\hat{s}_2(\vec{\gm})}^k$. These orbits 
are distinct when $\gm$ is separating and coincide when $\gm$ is non-separating. 
\end{theorem}

\begin{remark}\label{conjugacy discrete}Let us observe that in the discrete case the analogue 
of Theorem~\ref{conjugacy classes} follows from elementary topological considerations. 
Indeed, lifting the power $\tau_\gm^k\in\GG_{g,n}$ of a Dehn twist along the 
epimorphism $\GG_{g,n+1}\ra\GG_{g,n}$, induced by filling in the $(n+1)$-th puncture on 
$S_{g,n+1}$, amounts to removing a point from the Riemann surface $S_{g,n}$ outside the curve 
$\gm$ and then moving the puncture to the marked point $P_{n+1}$. So, modulo point-pushing 
homeomorphisms based at $P_{n+1}$, which correspond to the conjugacy action of $\Pi_{g,n}$, 
there is essentially one choice in case $\gm$ is non-separating and two in case $\gm$ is 
separating. 
\end{remark}

It is enough to prove Theorem~\ref{conjugacy classes} for a discrete s.c.c. $\gm\in\cL_{g,n}$ and for 
$k\in\Z\ssm\{0\}$. In order to do this, we are going to use the short exact sequence $(2.1)$ of 
Section~\ref{levels}.

Let $K$ be a finite index characteristic proper subgroup of $\Pi_{g,n}$ and
let $\GG_{g,[n]+1}$ be the group of homotopy classes of orientation preserving homeomorphisms
of $S_{g,n}$ which fix the base point $P_{n+1}$ of the fundamental group $\Pi_{g,n}$. Since $K$ is 
a characteristic subgroup of $\Pi_{g,n}$, it is induced a natural representation 
$\GG_{g,[n]+1}\ra\out(K)$, such that the image of the subgroup $\Pi_{g,n}$ of $\GG_{g,[n]+1}$
is identified, by the Nielsen realization Theorem, with the covering transformation group $G_K$ of 
the covering $p_K\co S_K\ra S_{g,n}$ associated to the subgroup $K$. 

By the Nielsen realization Theorem, with the notations of Section~\ref{levels}, there is then
a commutative diagram with exact rows and surjective vertical maps:
$$\begin{array}{ccccc}
1\ra&\Pi_{g,n}\ra&\GG_{g,[n]+1}\ra&\GG_{g,[n]}\ra 1&\\
&\da\hspace{0.5 cm}&\da&\parallel\hspace{1 cm}&\hspace{1cm}(5.1)\\
1\ra &G_K\ra &N_{\GG(S_K)}(G_K)\ra&\GG_{g,[n]}\ra 1.&
\end{array}$$

\begin{definition}\label{pseudo-twist}An element of the group $N_{\GG(S_K)}(G_K)$ is called
a pseudo-twist if it is the image of a Dehn twist of $\GG_{g,[n]+1}$ via the natural epimorphism
$\GG_{g,[n]+1}\tura N_{\GG(S_K)}(G_K)$ of diagram $(5.1)$.
\end{definition}

Let us observe that, for any s.c.c. $\gm$ on $S_{g,n}$, the surface $S_K\ssm p_K^{-1}(\gm)$ is 
disconnected and such that every circle in $p_K^{-1}(\gm)$ is a boundary component of two 
distinct connected components (cf. the proof of Lemma~3.10 \cite{sym}). Then, it holds:

\begin{proposition}\label{char-pseudo}\begin{enumerate}
\item An element $\upsilon_\gm\in N_{\GG(S_K)}(G_K)$ which
lifts the $k$-th power of a Dehn twist $\tau^k_\gm\in\GG_{g,[n]}$, for $k\in\Z\ssm\{0\}$, is the 
$k$-th power of a pseudo-twist if and only if it restricts to the identity on at least one of the 
components of the surface $S_K\ssm p_K^{-1}(\gm)$.
\item Let $S$ and $S'$ be the two connected components of the surface $S_K\ssm p_K^{-1}(\gm)$
sharing the boundary component $\td{\gm}\subset p_K^{-1}(\gm)$. Let then 
$\upsilon_{\td{\gm}}$ and $\upsilon_{\td{\gm}}'$ be the $k$-th powers of pseudo-twists which lift the 
given power of Dehn twist $\tau^k_\gm$ and which restrict to the identity, respectively, on $S$ and 
$S'$. Then, there is an element $\gm_\ast\in\Pi_{g,n}$, whose free homotopy class contains $\gm$, 
such that its image $\ol{\gm}_\ast$ in $G_K$ generates the stabilizer of the s.c.c. 
$\td{\gm}\subset S_K$ and it holds $\upsilon_{\td{\gm}}'\upsilon_{\td{\gm}}^{-1}=\ol{\gm}_\ast^k$.
\end{enumerate}
\end{proposition}

\begin{proof}Let us assume that the s.c.c. $\gm$ is non-separating, the other case can be treated
similarly. With the same notations and conventions of the proof of Lemma~\ref{closure}, let us
assume moreover that $\gm$ is disjoint from the loops 
$\alpha_1,\ldots\alpha_{g-1},\beta_1\ldots\beta_{g-1}, u_1,\ldots u_{n+2}$.
The natural representation $\GG_{g,[n]+1}\hookra\aut(\Pi_{g,n})$ is then such that the power of 
Dehn twist $\td{\tau}_\gm^k\in\GG_{g,[n]+1}$, which lifts the power of Dehn twist 
$\tau_\gm^k\in\GG_{g,[n]}$, acts trivially on the subgroup $\Pi_{g-1,n+2}$ of $\Pi_{g,n}$. 
In particular, $\td{\tau}_\gm^k$ acts trivially on the intersection $K\cap\Pi_{g-1,n+2}$.

Let $\td{P}$ be a point of $S_K$ lying above the base point $P_{n+1}$ of the fundamental group
$\Pi_{g,n}$ of $S_{g,n}$ and let us identify $K$ with the fundamental group of $S_K$ based at
$\td{P}$. The fundamental group of the connected component of $S_K\ssm p_K^{-1}(\gm)$ 
which contains the point $\td{P}$ clearly identifies with the intersection $K\cap\Pi_{g-1,n+2}$.
Hence, the power of Dehn twist $\td{\tau}_\gm^k$ induces the trivial action on this
connected component. Since all connected components of $S_K\ssm p_K^{-1}(\gm)$ 
are in a single orbit for the action of $G_K$, part $i.)$ of Proposition~\ref{char-pseudo} follows.

As to part $ii.)$, let us just observe that $\upsilon_{\td{\gm}}'\upsilon_{\td{\gm}}^{-1}$ is the image,
via the representation $\GG_{g,[n]+1}\ra N_{\GG(S_K)}(G_K)$, of the $k$-th power of the bounding
pair map associated to an element $\gm_\ast\in\Pi_{g,n}$ which identifies in
the group $\out(K)$ with the element of $G_K$ with the properties stated in part $ii.)$.

\end{proof}

Let $\check{N}_{\GG(S_K)}(G_K)$ be the closure of $N_{\GG(S_K)}(G_K)$ 
inside the profinite group $\cGG(S_K)$. This is the same profinite completion as that
induced by the procongruence completion $\cGG_{g,[n]+1}$ via the natural epimorphism
$\GG_{g,[n]+1}\tura N_{\GG(S_K)}(G_K)$. Hence, the diagram $(5.1)$ induces 
a commutative diagram with exact rows and surjective vertical maps:
$$\begin{array}{ccccc}
1\ra&\hP_{g,n}\ra&\cGG_{g,[n]+1}\ra&\cGG_{g,[n]}\ra 1&\\
&\da\hspace{0.5 cm}&\da&\parallel\hspace{1 cm}&\hspace{1cm}(5.2)\\
1\ra &G_K\ra &\check{N}_{\GG(S_K)}(G_K)\ra&\cGG_{g,[n]}\ra 1.&
\end{array}$$
Moreover, the upper row of the diagram is isomorphic to the inverse limit of the bottom rows,
when $K$ varies among all finite index characteristic subgroup of $\Pi_{g,n}$.
Theorem~\ref{conjugacy classes} then follows from the lemma:

\begin{lemma}\label{profinite pseudo}Let $\tau_\gm^k\in\cGG_{g,[n]}$, for $k\in\Z\ssm\{0\}$, be 
the power of a discrete Dehn twist and let $\up_\gm$ be a lift of this element to 
$\check{N}_{\GG(S_K)}(G_K)$, which then belongs to the discrete subgroup 
$N_{\GG(S_K)}(G_K)$. Then, $\up_\gm$ belongs to the closure of the set of $k$-th powers of 
pseudo-twists in the profinite group $\check{N}_{\GG(S_K)}(G_K)$ 
if and only if it is the $k$-th power of a pseudo-twist.
\end{lemma}

\begin{proof}The inverse image of $\GG_{g,[n]}$, considered as a subgroup of $\cGG_{g,[n]}$, via the
epimorphism $\check{N}_{\GG(S_K)}(G_K)\ra\cGG_{g,[n]}$ is the discrete subgroup
$N_{\GG(S_K)}(G_K)$. Therefore, any lift in the profinite group $\check{N}_{\GG(S_K)}(G_K)$ of 
a discrete element of the profinite group $\cGG_{g,[n]}$ is also discrete.

Let $\upsilon_\gm\in N_{\GG(S_K)}(G_K)$ be a lift of $\tau_\gm^k\in\GG_{g,[n]}$ which
does not restrict to the identity on any component of the surface $S_K\ssm p_K^{-1}(\gm)$.
By Proposition~\ref{char-pseudo}, in order to prove the lemma, it is enough to show that 
$\upsilon_\gm$ is not in the closure in $\check{N}_{\GG(S_K)}(G_K)$ of the set of $k$-th powers 
of pseudo-twists of $N_{\GG(S_K)}(G_K)$. Let us then assume, on the contrary, that $\up_\gm$ 
belongs to such a closure and let us show how this leads to a contradiction.  

Let $s$ be the smallest positive integer such that $\gm^s\in K$. Then, $\tau_\gm^s\in\GG_{g,[n]}$
lifts, canonically, in $N_{\GG(S_K)}(G_K)$ to the product, which we denote by $\xi_\gm$, of the Dehn 
twists about the s.c.c.'s contained in $p_K^{-1}(\gm)$. This element is in the centralizer of the finite
group $G_K$. It follows that the $\check{N}_{\GG(S_K)}(G_K)$-conjugacy orbit of 
$\xi_\gm^k$ is mapped bijectively, by the epimorphism $\check{N}_{\GG(S_K)}(G_K)\ra\cGG_{g,[n]}$, 
onto the $\cGG_{g,[n]}$-conjugacy orbit of $\tau_\gm^{sk}$. 

By the above assumption, $\up_\gm^s$ is in the closure in $\check{N}_{\GG(S_K)}(G_K)$ of the set 
of $sk$-th powers of pseudo-twists of $N_{\GG(S_K)}(G_K)$. Hence,
$\up_\gm^s$ is in the $\check{N}_{\GG(S_K)}(G_K)$-conjugacy orbit of $\xi_\gm^k$. Since both 
elements project to $\tau_\gm^{sk}$, it holds $\up_\gm^s=\xi_\gm^k$.

For a given commutative ring $A$ with $1$, let us denote the homology group $H_1(\ol{S}_K,A)$ 
simply by $H_A$ and by $\langle\_,\_\rangle_A$ its standard symplectic form. Let then
$\mathfrak{sp}(H_A)$ be the Lie algebra of the symplectic group $\Sp(H_A)$, i.e. the $A$-module 
of endomorphisms $\varphi$ of $H_A$ satisfying the identity
$\langle\varphi(x),y\rangle_A+\langle x,\varphi(y)\rangle_A=0$, for all $x,y\in H_A$.

For $\varphi\in\mathfrak{sp}(H_A)$, the assignment $\varphi\mapsto\langle\_,\varphi(\_)\rangle_A$
defines a natural isomorphism:
$$\mathfrak{sp}(H_A)\sr{\sim}{\ra}\mathrm{L}^2_s(H_A),$$
where $\mathrm{L}^2_s(H_A)$ is the $A$-module of symmetric bilinear forms on $H_A$. There is
also a natural isomorphism $\mathrm{L}^2_s(H_A)\cong\mathrm{Sym}^2 H_A^\ast$, where 
$H_A^\ast:=\hom_A(H_A,A)$ is the dual $A$-module which can be identified with the cohomology 
group $H^1(\ol{S}_K,A)$. 

The Poincar\'e dual of an element $a\in H_A$ is the $A$-linear map 
$a^\vee:=\langle\_, a\rangle_A\in H_A^\ast$ and, for a submodule $M$ of $H_A$, its Poincar\'e dual 
$M^\vee$ is defined to be the submodule $\{m^\vee |m\in M\}$ of $H_A^\ast$. In contrast with 
$(\_)^\ast$, this is a covariant functor. 

Let us denote also by $\langle\_,\_\rangle_A$ the symplectic form induced 
on $H_A^\ast$ by the standard symplectic form on $H_A$, i.e. 
$\langle a^\vee,b^\vee\rangle_A:=\langle b,a\rangle_A$. The Poincar\'e duals of elements 
and submodules of $H_A^\ast$ are defined as above. For $a\in H_A$ 
or $a\in H_A^\ast$, the identity $(a^\vee)^\vee=a$ then holds, where $H_A$ is identified
with $H_A^{\ast\ast}$ by the usual canonical assignment $a\mapsto[f\mapsto f(a)]$.

An element $f\in\GG(S_K)$ induces a symplectic automorphism of $H_A$, which we also denote by 
$f$. Let us define the logarithm $\log f$ to be the symmetric bilinear form on $H_A$ associated 
to the endomorphism $f-\mathrm{id}_{H_A}$.

For $\alpha$ a s.c.c. on the closed Riemann surface $\ol{S}_K$, let us denote by $\vec{\alpha}$ the
cycle determined in $H_A$ by the s.c.c. $\alpha$ with a given orientation.
From simple direct computations (cf. \S 1 in \cite{L}), it follows that, for the Dehn twist $\tau_\alpha$, 
it holds $\log\tau_\alpha=\vec{\alpha}^\vee\otimes\vec{\alpha}^\vee$. In particular, the logarithm of
$\tau_\alpha$ does not depend on the chosen orientation on $\alpha$.

It is possible to recover the submodule $\langle\vec{\alpha}\rangle$ of $H_A$ generated by the cycle 
$\vec{\alpha}$ from the logarithm $\vec{\alpha}^\vee\otimes\vec{\alpha}^\vee$ of the Dehn twist 
$\tau_\alpha$ as the Poincar\'e dual of the dual of the cokernel of the corresponding symmetric 
bilinear form:
$$\langle\vec{\alpha}\rangle=
\left(\left(H_A/\ker\vec{\alpha}^\vee\otimes\vec{\alpha}^\vee\right)^\ast\right)^\vee.$$

For a given symmetric bilinear form $\varphi\in\mathrm{Sym}^2 H_A^\ast$, in order to simplify
the terminology, we call the Poincar\'e dual of the dual of its cokernel, which is a submodule
of $H_A$, the {\it core} of the form $\phi$. So that, we can write 
$\mathrm{core}\,(\vec{\alpha}^\vee\otimes\vec{\alpha}^\vee)=\langle\vec{\alpha}\rangle$.

More generally, for $A=\Z$, it holds:

\begin{lemma}\label{mochizuki}For $c_1,\ldots,c_h\in\Z^+$ and 
$\alpha_1,\ldots\alpha_h\in H_\Z$, it holds:
$$\langle\vec{\alpha_1},\ldots,\vec{\alpha}_h\rangle=\mathrm{core}\,\left(
\sum_{i=1}^h c_i\,\vec{\alpha_i}^\vee\otimes\vec{\alpha_i}^\vee\right).$$
\end{lemma}
\begin{proof}The lemma follows from the identity
$\ker(\sum_{i=1}^h c_i\,\vec{\alpha_i}^\vee\otimes\vec{\alpha_i}^\vee)=
\bigcap_{i=1}^h\ker\vec{\alpha_i}^\vee\otimes\vec{\alpha_i}^\vee$, which we are going 
to prove below, and the series of identities:
$$\left(\left(H_\Z\left/\bigcap_{i=1}^h \ker\vec{\alpha_i}^\vee\otimes\vec{\alpha_i}^\vee\right.
\right)^\ast\right)^\vee=\langle\bigcup_{i=1}^h\left(H_\Z/\ker\vec{\alpha_i}^\vee
\otimes\vec{\alpha_i}^\vee\right)^\ast \rangle^\vee=
\langle\vec{\alpha_1},\ldots,\vec{\alpha}_h\rangle.$$

Let $\varphi:=\sum_{i=1}^h c_i\,\vec{\alpha_i}^\vee\otimes\vec{\alpha_i}^\vee$. Then, the symmetric
bilinear form $\varphi$ induces a positive definite form on the quotient $H_\Z/\ker\varphi$. Since
$\ker\varphi$ is a primitive sublattice of $H_\Z$, there is a splitting 
$H_\Z=\ker\varphi\oplus M$ such that $\varphi$ restricts to a positive definite symmetric bilinear
form on $M$. It holds $H_\Q\cong(\ker\varphi\otimes\Q)\oplus (M\otimes\Q)$ and $\varphi$
extends to a symmetric bilinear form $\varphi'$ on this $\Q$-vector space which restricts
to a positive definite form on $M\otimes\Q$ and such that
$\ker\varphi'=\ker\varphi\otimes\Q$ and $\ker\varphi=\ker\varphi'\cap H_\Z$.

Let us denote by $\xi_i$ the extension of the symmetric bilinear form 
$\vec{\alpha_i}^\vee\otimes\vec{\alpha_i}^\vee$ to $H_\Q$, for $i=1,\ldots,h$. Then, it holds
$\varphi'=\sum_{i=1}^h c_i\,\xi_i$ and 
$\ker\vec{\alpha_i}^\vee\otimes\vec{\alpha_i}^\vee=\ker\xi_i\cap H_\Z$, for $i=1,\ldots,h$. 
By elementary linear algebra, it holds $\ker\varphi'=\bigcap_{i=1}^h \ker\xi_i$ and this 
implies the identity claimed above.

\end{proof}

Let us denote by $(\log\xi_\gm^k)_p$ the reduction mod $p$ of the symmetric bilinear form 
$\log\xi_\gm^k$. The latter induces a non-degenerate form on the 
quotient $H_\Z/\ker(\log\xi_\gm^k)$. Therefore, there is a non-empty open set $U$ of $\Spec(\Z)$ 
such that, for all $p\in U$, the symmetric bilinear form $(\log\xi_\gm^k)_p$ induces 
a non-degenerate form on the mod $p$ reduction 
$(H_{\Z}/\ker(\log\xi_\gm^k))\otimes\F_p$, which is then isomorphic to
$H_{\F_p}/\ker(\log\xi_\gm^k)_p$. It follows that, for all primes $p\in U$, it holds 
$\ker(\log\xi_\gm^k)_p\cong\ker(\log\xi_\gm^k))\otimes\F_p$.

For any given s.c.c. $\alpha$ on $S_{g,n}$, let us denote by $H_\alpha$ the primitive submodule of 
$H_\Z$ generated by the cycles supported in $p_K^{-1}(\alpha)$. For a prime $p$, let $\F_p$ be 
the field with $p$ elements, let then $H_{\alpha,p}$ be the image of $H_\alpha$ in the $\F_p$-vector 
space $H_{\F_p}$ via the natural epimorphism $H_\Z\tura H_{\F_p}$. There is a natural isomorphism 
$H_{\alpha,p}\cong H_\alpha\otimes\F_p$. 

By Lemma~\ref{mochizuki}, $H_\gm$ is the core of the bilinear form
$\log\xi_\gm^k=k\log\xi_\gm$ on $H_\Z$. 
It follows that, for all primes $p\in U$, the core of the mod $p$ reduction $(\log\up_\gm^s)_p$ of the 
bilinear form $\log\up_\gm^s=\log\xi_\gm^k$ is the subspace $H_{\gm,p}$ of $H_{\F_p}$.

By hypothesis, there is  a pseudo-twist $\td{\tau}_\delta$ of $N_{\GG(S_K)}(G_K)$, lifting a
Dehn twist $\tau_\delta\in\GG_{g,[n]}$, such that $\td{\tau}_\delta^k$ and $\up_\gm$ map to
the same element of the symplectic group $\Sp(H_{\F_p})$, for some $p\in U$ with $p>2$. 
In particular, the core of the bilinear form $(\log\td{\tau}_\delta^{ks})_p$ on $H_{\F_p}$ is the 
subspace $H_{\gm,p}$. Since $H_{\delta,p}$ is contained in the core of 
$(\log\td{\tau}_\delta^{ks})_p$ and is of the same dimension, i.e. the dimension of $H_{\gm,p}$, 
it follows that it holds $H_{\delta,p}=\ker(\log\td{\tau}_\delta^{ks})_p=H_{\gm,p}$.

For a given s.c.c. $\alpha$ on $S_{g,n}$, let $\Sg_\alpha$ be the dual graph of the degeneration of 
the Riemann surface $\ol{S}_K$ obtained contracting all s.c.c.'s contained in $p_K^{-1}(\alpha)$. 
There is then an epimorphism $\Phi_\alpha\co H_{\Z}\tura H_1(\Sg_\alpha,\Z)$ which is the Poincar\'e 
dual of the dual of the inclusion $H_{\alpha}\subset H_{\Z}$ (cf. Proposition~1 and 1bis in \cite{Br}). 
The epimorphism $\Phi_\alpha$ contains $H_{\alpha}$ in its kernel and induces an epimorphism 
$\ol{\Phi}_\alpha\co H_{\Z}/H_{\alpha}\tura H_1(\Sg_\alpha,\Z)$. The standard symplectic form on 
$H_\Z$ induce a non-degenerate symplectic form on $K_\alpha:=\ker\ol{\Phi}_\alpha$.

The above statements hold for any system of coefficients for the homology. 
For a prime $p$, let then $\ol{\Phi}_{\alpha,p}\co H_{\F_p}/H_{\alpha,p}\tura H_1(\Sg_\alpha,\F_p)$ 
be the corresponding epimorphism with $\F_p$-coefficients and let
$K_{\alpha,p}:=\ker\ol{\Phi}_{\alpha,p}$.

With the above notations, it holds $K_{\gm,p}=K_{\delta,p}$ and there is then a natural 
symplectic isomorphism $K_\gm\otimes\F_p\cong K_\delta\otimes\F_p$. This isomorphism lifts to a 
symplectic isomorphism $K_\gm\cong K_\delta$. If we identify the latter two $\Z$-modules by
means of this isomorphism, the elements $\upsilon_\gm$ and $\td{\tau}_\delta^k$ both act on the 
$\Z$-module $K_\gm$ and have the same reduction modulo $p$. 

Let $\ol{S}_\gm$ be the disconnected closed Riemann surface obtained filling in all the punctures 
of $S_K\ssm p_K^{-1}(\gm)$. Then, $H_1(\ol{S}_\gm,\Z)=K_\gm$ and the lift $\upsilon_\gm$ acts 
on $\ol{S}_\gm$ as a finite order homeomorphism which does not restrict to the identity on any 
component of $\ol{S}_\gm$. 

Let then $\ol{S}'_\gm$ be the quotient of $\ol{S}_\gm$ by the action of the 
finite cyclic group generated by $\upsilon_\gm$. It follows that the invariant submodule 
$(K_\gm)^{\upsilon_\gm}$ is naturally isomorphic to the homology group $H_1(\ol{S}'_\gm,\Z)$ and 
hence that $(K_\gm)^{\upsilon_\gm}$ does not contain any primitive non-degenerate symplectic 
submodule of $K_\gm$.  

By Proposition~\ref{char-pseudo}, the action induced by $\td{\tau}_\delta^k$ on $\ol{S}_\delta$ 
restricts to the identity on at least a component of this surface and then preserves a primitive 
non-degenerate symplectic submodule of $K_\delta$. Therefore, the action induced by 
$\td{\tau}_\delta^k$ on $K_\gm$, by means of the above identification, is not conjugated to that of 
$\upsilon_\gm$ but induces the same action on $K_\gm\otimes\F_p$. Lemma~\ref{profinite pseudo} 
is then a consequence of the following lemma which gives a contradiction:

\begin{lemma}Let $C_1$ and $C_2$ two finite cyclic subgroups of $\Sp_{2g}(\Z)$ which are not
conjugated inside this group. Then, for all integers $m\geq 3$, their images $\ol{C}_1$ and 
$\ol{C}_2$ in the quotient group $\Sp_{2g}(\Z/m)$ are also non-conjugated.
\end{lemma}

\begin{proof}This fact is possibly well known but is, at least, an easy consequence of the theory of
abelian varieties with automorphisms. Let $\cA_1$ and $\cA_2$ be the irreducible substacks of 
$\cA_g$, the moduli stack of complex principally polarized abelian varieties of dimension $g$, 
parametrizing abelian varieties with an automorphism conjugated, respectively, to $C_1$ and 
$C_2$. Since $C_1$ and $C_2$ are not conjugated in $\Sp_{2g}(\Z)$, it follows that 
$\cA_1\neq\cA_2$.

Let $\cA_g^{(m)}$ be the moduli stack of complex principally polarized abelian varieties of 
dimension $g$ with a level $m\geq 3$ structure. The natural morphism $p\co\cA_g^{(m)}\ra\cA_g$ 
is a Galois covering with deck transformation group $\Sp_{2g}(\Z/m)$ and the irreducible 
components of $p^{-1}(\cA_1)$ and $p^{-1}(\cA_2)$ are the fixed point loci of the conjugates inside 
$\Sp_{2g}(\Z/m)$ of the subgroups $\ol{C}_1$ and $\ol{C}_2$, respectively. If these two groups 
were in the same conjugacy class, it would follow that $\cA_1=\cA_2$, which is not the case.

\end{proof}\end{proof}

By Proposition~\ref{extend}, there are continuous maps
$\hat{s}^k_{i\ast}\co\hL^\ast_{g,n}\ra\cGG_{g,n+1}$, for $i=1,2$ and $k\in\ZZ$,  
sending a profinite loop $\vec{\gm}\in\hL^\ast_{g,n}$ to the power of Dehn twist 
$\tau_{\hat{s}_i(\vec{\gm})}^k\in\cGG_{g,n+1}$.

For $\gm\in\hL_{g,n}$ a given profinite s.c.c., let us then describe in a more 
precise way the set of all pairs of $k$-th powers of Dehn twists 
$(\tau^{k}_{\hat{s}_1(\vec{\alpha})},\tau^{k}_{\hat{s}_2(\vec{\alpha})})$, for 
$\vec{\alpha}\in\cL^\ast_{g,n}$ such that $\tau_{\hat{q}(\vec{\alpha})}^k=\tau_\gm^k$,
where $\hat{q}\co\hat{\cL}^\ast_{g,n}\tura\hat{\cL}_{g,n}$ is the natural continuous 
surjective map.

Let us define the $\cGG_{g,n+1}$-equivariant continuous map
$\hat{s}^k_{\ast}\co\hL^\ast_{g,n}\ra\cGG_{g,n+1}\times\cGG_{g,n+1}$
which assigns to a profinite oriented pointed s.c.c. $\vec{\gm}\in\hL^\ast_{g,n}$ the pair of powers 
of profinite Dehn twists $(\tau^{k}_{\hat{s}_1(\vec{\gm})},\tau^{k}_{\hat{s}_2(\vec{\gm})})$,
where $\cGG_{g,n+1}$ act on the product $\cGG_{g,n+1}\times\cGG_{g,n+1}$
by conjugation via the diagonal embedding.
The map $\hat{s}^k_{\ast}$ is injective, since it factors the injective map which sends
$\vec{\gm}\in\hL^\ast_{g,n}$ to the product 
$\tau^{k}_{\hat{s}_1(\vec{\gm})}\tau^{-k}_{\hat{s}_2(\vec{\gm})}\in\cGG_{g,n+1}$.

\begin{theorem}\label{lifts}For $2g-2+n>0$ and $k\in\ZZ\ssm\{0\}$, let $\gm\in\hL_{g,n}$ 
be a given profinite s.c.c.. Then, the set formed by the ordered pairs 
$(\tau^k_{\hat{s}_1(\vec{\alpha})},\tau^k_{\hat{s}_2(\vec{\alpha})})$, for 
$\vec{\alpha}\in\hL^\ast_{g,n}$, such that $\tau_{\hat{q}(\vec{\alpha})}^k=\tau_\gm^k$, consists of 
two $\hP_{g,n}$-orbits, each corresponding to an orientation of $\gm$, where 
the group $\hP_{g,n}$ acts diagonally on such pairs by conjugation.
\end{theorem}
\begin{proof}As usual, it is enough to prove the theorem in case $\gm$ is the class of a s.c.c.
on $S_{g,n}$. 

By Theorem~\ref{conjugacy classes}, the set of pairs of powers of profinite Dehn twists of the form 
$(\tau^k_{\hat{s}_1(\vec{\alpha})},\tau^k_{\hat{s}_2(\vec{\alpha})})$, for 
$\vec{\alpha}\in\hL^\ast_{g,n}$, lying above the element $\tau_\gm^k\in\GG_{g,n}$, is the inverse 
limit of the set of pairs of powers of pseudo-twists described in ii.) Proposition~\ref{char-pseudo}. 
These pairs fit at most in two conjugacy classes, corresponding to the possible 
orientations of $\gm$. Therefore, the same holds for the pairs of powers of Dehn twists lying above 
$\tau_\gm^k$. 

On the other hand, since the two lifts $\vec{\gm}$ and $\vec{\gm}^{-1}$ of $\gm$ in 
$\hL^\ast_{g,n}$ are not conjugated in the profinite group $\hP_{g,n}$, it follows that there are 
at least two such conjugacy classes. 

\end{proof}

We can now prove Theorem~\ref{parametrize}. For ${\alpha},{\beta}\in\hL_{g,n}$ such that
$\tau_\alpha^{k}=\tau_\beta^{k}$, by Theorem~\ref{lifts}, there are lifts 
$\vec{\alpha},\vec{\beta}\in\hL^\ast_{g,n}$ such that the ordered
pairs $(\tau^{k}_{\hat{s}_1(\vec{\alpha})},\tau^{k}_{\hat{s}_2(\vec{\alpha})})$ and 
$(\tau^{k}_{\hat{s}_1(\vec{\beta})},\tau^{k}_{\hat{s}_2(\vec{\beta})})$ are in the same 
$\hP_{g,n}$-orbit. Since the injective map 
$\hat{s}^k_{\ast}\co\hL^\ast_{g,n}\hookra\cGG_{g,n+1}\times\cGG_{g,n+1}$ is 
$\hP_{g,n}$-equivariant, it follows that $\vec{\alpha}$ and $\vec{\beta}$ are 
conjugated in $\hP_{g,n}$, i.e. $\alpha=\beta$.

\begin{remark}\label{separability}From the identity $\hL_{g,n}^\ast\cap\GG_{g,n+1}=\cL_{g,n}^\ast$, 
it follows that, if we denote by $\cD^k$, for $k\in\ZZ$, the set of $k$-th powers of Dehn twists of the 
Teichm\"uller group $\GG_{g,n}$ and by $\check{\cD}^k$ its closure in the procongruence completion
$\cGG_{g,n}$, it holds $\check{\cD}^k\cap\GG_{g,n}=\cD^k$, if $k\in\Z$, otherwise this intersection is
empty. Together with Theorem~\ref{parametrize}, this implies that, for $\gm,\gm'\in\hat{\cL}$ and 
$k,k'\in\ZZ\ssm\{0\}$, it holds $\tau_\gm^k=\tau_{\gm'}^{k'}$ if and only if $\gm=\gm'$ and $k=k'$.
\end{remark}

\section{Centralizers of profinite Dehn twists in $\cGG_{g,n}$}\label{centralizers}
An immediate consequence of Theorem~\ref{parametrize} and Theorem~\ref{stab pro-curves} is 
a description of centralizers of powers of profinite Dehn twists in the procongruence Teichm\"uller 
group:

\begin{corollary}\label{centralizer twist}For $2g-2+n>0$, let $\gm\in\hL_{g,n}$ and $k\in\ZZ\ssm{0}$.
Then, the centralizer in $\cGG_{g,n}$ of the element $\tau_\gm^k$
coincides with the stabilizer $\cGG_\gm$ of the profinite s.c.c. $\gm$ in $\hL_{g,n}$. 
\end{corollary}

The center of the procongruence Teichm\"uller group $\cGG_{g,n}$ and the centralizers of its 
open subgroups are then also determined:

\begin{corollary}\label{center-free}Let $2g-2+n>0$. For every open subgroup $U$ of 
$\cGG_{g,n}$, it holds:
$$Z(\cGG_{g,n})=Z_{\cGG_{g,n}}(U)=Z(\GG_{g,n}).$$
Thus, all these groups are trivial for $(g,n)\neq(1,1),(2,0)$ and, otherwise, they are 
generated by the hyperelliptic involution. In particular, by Theorem~\ref{stab pro-curves},  
for $\vec{\gm}$ an oriented s.c.c. on the Riemann surface $S_{g,n}$, it holds 
$Z(\cGG_{\vec{\gm}})=Z(\GG_{\vec{\gm}})$.
\end{corollary}
\begin{proof}For $g\leq 2$, this is just Proposition~3.2 in \cite{Hyp}.
Let us then proceed by induction on the genus.
For $g\geq 3$, let $\gm$ be a s.c.c. on $S_{g,n}$ bounding an unpunctured genus $1$ subsurface 
of $S_{g,n}$. By Corollary~\ref{centralizer twist}, for every $k\in\ZZ\ssm{0}$, the centralizer of the 
Dehn twist $\tau_\gm^k$ is the stabilizer $\cGG_\gm$. By Theorem~\ref{stab pro-curves}, $\cGG_\gm$ 
is the closure of the discrete stabilizer $\GG_\gm=\GG_{\vec{\gm}}$. In particular, it holds 
$\cGG_\gm=\cGG_{\vec{\gm}}$. By Theorem~\ref{stab pro-curves}, there is then a short exact 
sequence:
$$1\ra\ZZ\cdot\tau_{\gm}\ra\cGG_\gm\ra\cGG_{g_1,n_1}\times\cGG_{1,1}\ra 1.$$

From the induction hypothesis, it then follows that the centralizer $Z_{\cGG_\gm}(\cGG_\gm\cap U)$ 
is the closure in $\cGG_\gm$ of the center of the discrete subgroup $\GG_\gm$.
Therefore, the centralizer $Z_{\cGG_\gm}(\cGG_\gm\cap U)$ is the closed cyclic 
subgroup generated by the half Dehn twist $\tau_\gm^{1/2}$ about $\gm$ which restricts to the 
identity on the subsurface $S_{g_1,n_1}$.

The open subgroup $U$ of $\cGG_{g,n}$ contains a power of $\tau_\gm$. Hence, it holds
$Z_{\cGG_{g,n}}(U)\leq\cGG_\gm$ and then 
$Z_{\cGG_{g,n}}(U)\leq Z_{\cGG_\gm}(\cGG_\gm\cap U)$.
So, the centralizer of $U$ in $\cGG_{g,n}$ is topologically generated by a power of the half 
Dehn twist $\tau_\gm^{1/2}$. But no non-trivial power of $\tau_\gm^{1/2}$ is in the center of 
$\GG_{g,n}\cap U$. Therefore, it holds $Z_{\cGG_{g,n}}(U)=\{1\}$.

\end{proof}

Let $\gm$ be a s.c.c. on $S_{g,n}$. It is reasonable to expect that the set of profinite Dehn twists of 
$\cGG_{g,n}$, which are contained in the stabilizer $\cGG_\gm$, is the closure of the set of Dehn 
twists contained in the discrete stabilizer $\GG_\gm$. In virtue of Remark~\ref{link}, another way to 
state this is to say that the profinite Dehn twists of $\cGG_{g,n}$, contained in $\cGG_\gm$, are 
parametrized by the vertices of the closed star of $\gm$ in the complex of profinite curves 
$L(\hP_{g,n})$. More precisely, it holds:

\begin{theorem}\label{Dehn stabilizer}For $2g-2+n>0$, let $\alpha,\beta\in\hL_{g,n}$ and let
$\ol{\mathrm{St}}(\beta)$ be the closed star of $\beta$ inside 
the complex of profinite curves $L(\hP_{g,n})$.
\begin{enumerate}
\item If $\alpha\in\ol{\mathrm{St}}(\beta)_0$, then 
$\cGG_\beta\cap\ZZ\cdot\tau_\alpha=\ZZ\cdot\tau_\alpha$.
\item If $\alpha\notin\ol{\mathrm{St}}(\beta)_0$, then $\cGG_\beta\cap\ZZ\cdot\tau_\alpha=\{1\}$.
\end{enumerate}
\end{theorem}

\begin{proof}The complex of profinite curves $L(\hP_{g,n})$ parametrizes the nerve of the 
D--M boundary $\dd\tcM$ of the profinite covering $\tcM$ of $\ccM_{g,n}$ (cf. the proof of 
Theorem~\ref{curve-complex}) while the profinite set of vertices of the link 
$\mathrm{Lk}(\beta)$ parametrizes the boundary components of the irreducible component 
$\td{\delta}_\beta$ of the boundary $\dd\tcM$, associated to the profinite s.c.c. $\beta\in\hL$, 
i.e. the irreducible components of the locus where $\td{\delta}_\beta$ intersects the other 
irreducible components of $\dd\tcM$. 

The theorem then is equivalent to the claim that a non-trivial power 
$\tau_\alpha^h$ of a profinite Dehn twist of $\cGG_{g,n}$, for $\alpha\neq\beta\in\hL$, 
stabilizes $\td{\delta}_\beta$, if and only if, the 
irreducible component $\td{\delta}_\alpha$, associated to $\alpha$, intersects $\td{\delta}_\beta$, 
i.e, if and only if, the profinite Dehn twist $\tau_\alpha$ generates the inertia group of a 
boundary component of $\td{\delta}_\beta$.

The analogous statement is well known when $\tau_\alpha$ is a Dehn twist of $\GG_{g,n}$ and 
$\ol{T}_\beta$ is the boundary component of the Bers bordification 
$\ol{T}_{g,n}$ associated to a discrete s.c.c. $\beta$. It can be reformulated by saying that the 
natural map $\ol{T}_\beta\ra\ol{T}_{g,n}/\langle\tau_\alpha\rangle$ is either an embedding or 
factors through an embedding 
$\ol{T}_\beta/\langle\tau_\alpha\rangle\hookra\ol{T}_{g,n}/\langle\tau_\alpha\rangle$ and, in the latter
case, the subgroup $\langle\tau_\alpha\rangle$ of $\GG_{g,n}$ identifies with the inertia group of 
a boundary component of $\ol{T}_\beta$.

Let us describe explicitly the quotient $\tcM/\ol{\langle\tau_\alpha\rangle}$. 
Let $\{\GG^\ld\}_{\ld\in\Ld}$ be the tower of all geometric levels of the Teichm\"uller group 
$\GG_{g,n}$. There is a series of natural isomorphisms:
$$\tcM=\lim\limits_{\sr{\textstyle\st\longleftarrow}{\sst\ld\in\Ld}}\,\ol{T}_{g,n}/\GG^\ld\cong
\lim\limits_{\sr{\textstyle\st\longleftarrow}{\sst\ld\in\Ld}}\,\left.[\ol{T}_{g,n}\times(\GG_{g,n}/\GG^\ld)]
\right/\GG_{g,n}=\left.(\ol{T}_{g,n}\times\cGG_{g,n})\right/\GG_{g,n},$$ 
where an element $f\in\GG_{g,n}$ acts on the product $\ol{T}_{g,n}\times\cGG_{g,n}$ 
by the formula $f\cdot(x,h)=(f\cdot x,fh)$. Then, it holds:
$$\left.\tcM/\ol{\langle\tau_\alpha\rangle}\cong[\ol{T}_{g,n}\times
(\cGG_{g,n}/\ol{\langle\tau_\alpha\rangle})]\right/\GG_{g,n}.$$

Let us identify the Bers bordification $\ol{T}_{g,n}$ with a subspace of the covering $\tcM$.
For every $f\in\cGG_{g,n}$, the translation map $\phi_f\co\ol{T}_{g,n}\hookra\tcM$, defined by 
$x\mapsto f\cdot x$, corresponds, via the above isomorphism, to the natural injective map 
$\ol{T}_{g,n}\times\{f\}\hookra(\ol{T}_{g,n}\times\cGG_{g,n})/\GG_{g,n}$. 
In particular, the natural embedding of $\ol{T}_{g,n}$ inside $\tcM$ corresponds to the map $\phi_1$. 
The subspaces $\phi_f(\ol{T}_{g,n})$ are the analytic leafs of the space $\tcM$. 

For a given $\beta\in\hL$, let $\gm\in\cL$ and $h\in\cGG_{g,n}$ be such that $h(\gm)=\beta$.
Arguing as above, it follows that the boundary component $\td{\delta}_\beta$ of 
$\tcM$ is isomorphic to the quotient $(\phi_h(\ol{T}_\gm)\times\cGG_\beta)/h\GG_\gm h^{-1}$.
Then, for $f\in\cGG_\beta$, the natural injective maps 
$$\ol{T}_\gm\times\{fh\}\hookra\left.\ol{T}_{g,n}\times\cGG_{g,n}\right/\GG_{g,n}$$ 
describe the analytic leafs of the boundary component $\td{\delta}_\beta$ of $\tcM$.

Let $\psi_f$ be the composition of $\phi_f$, for $f\in\cGG_{g,n}$, with the quotient map
$\tcM\ra\tcM/\ol{\langle\tau_\alpha\rangle}$. Then, the analytic leafs of 
$\tcM/\ol{\langle\tau_\alpha\rangle}$ are given by the images $\psi_f(\ol{T}_{g,n})$, 
for $f\in\cGG_{g,n}$.

In order to describe these leafs, we need to determine the stabilizer of a left coset 
$f\cdot\ol{\langle\tau_\alpha\rangle}$ in the group $\cGG_{g,n}$, for $f\in\cGG_{g,n}$, for the 
action of its discrete subgroup $\GG_{g,n}$ by translations.

The stabilizer of the left coset $f\cdot\ol{\langle\tau_\alpha\rangle}$ for the action of the group 
$\cGG_{g,n}$ on itself by translations is clearly the conjugate subgroup
$f\ol{\langle\tau_\alpha\rangle}f^{-1}=\ol{\langle\tau_{f(\alpha)}\rangle}$.
So, the stabilizer for the action of $\GG_{g,n}$ is the intersection
$\GG_{g,n}\cap\ol{\langle\tau_{f(\alpha)}\rangle}$. By Remark~\ref{separability}, this intersection
equals $\langle\tau_{f(\alpha)}\rangle$, if $f(\alpha)\in\cL$, and is trivial if $f(\alpha)\notin\cL$.

The map $\psi_f\co\ol{T}_{g,n}\ra\tcM/\ol{\langle\tau_\alpha\rangle}$ is identified with the natural map:
$$\ol{T}_{g,n}\times\{f\}\ra\left.\ol{T}_{g,n}\times
(\cGG_{g,n}/\ol{\langle\tau_\alpha\rangle})\right/\GG_{g,n}.$$
Therefore, the map $\psi_f$ is injective for $f(\alpha)\notin\cL$ and factors through the embedding
$\phi_f(\ol{T}_{g,n})/\langle\tau_{f(\alpha)}\rangle\hookra\tcM/\ol{\langle\tau_\alpha\rangle}$,
for $f(\alpha)\in\cL$.

For a given $\beta\in\hL$, let, as above, $\gm\in\cL$ and $h\in\cGG_{g,n}$ be such that 
$h(\gm)=\beta$. For $f\in\cGG_\beta$, the natural map 
$$\ol{T}_\gm\times\{fh\}\ra\left.\ol{T}_{g,n}\times(\cGG_{g,n}/
\ol{\langle\tau_\alpha\rangle})\right/\GG_{g,n}$$ 
is then injective unless $fh(\alpha)\in\cL$ and the subgroup $\langle\tau_{\alpha}\rangle$ 
identifies with the inertia group of a boundary component of the analytic leaf 
$\ol{T}_\gm\times\{fh\}$ of $\td{\delta}_\beta$, in which case, the closed subgroup 
$\ol{\langle\tau_\alpha\rangle}$
identifies with the inertia group of a boundary component of $\td{\delta}_\beta$. 

We conclude that the natural map $\td{\delta}_\beta\ra\tcM/\ol{\langle\tau_\alpha\rangle}$ is either 
injective or factors through an embedding 
$\td{\delta}_\beta/\ol{\langle\tau_\alpha\rangle}\hookra\tcM/\ol{\langle\tau_\alpha\rangle}$, 
where the subgroup $\ol{\langle\tau_\alpha\rangle}$ identifies with the inertia group of a boundary 
component of $\td{\delta}_\beta$. This completes the proof of the theorem.

\end{proof}

Thanks to the above results, sets of commuting powers of profinite Dehn twists in the procongruence 
Teichm\"uller group and their centralizers can be characterized like in the discrete case:

\begin{corollary}\label{centralizer2}For $2g-2+n>0$, it holds:
\begin{enumerate}
\item A set of non-trivial powers of profinite Dehn twists 
$\{\tau_{\gm_0}^{h_0},\ldots,\tau_{\gm_k}^{h_k}\}$ of $\cGG_{g,n}$ 
consists of mutually commuting elements if and only if the set of profinite s.c.c.'s 
$\{\gm_0,\ldots,\gm_k\}$ is a simplex of $L(\hP_{g,n})$.

\item Let be given a $k$-simplex $\sg=\{\gm_0,\ldots,\gm_k\}\in L(\hP_{g,n})$
and $h_0,\ldots,h_k\in\ZZ\ssm\{0\}$. 
Then, the centralizer $Z_{\cGG_{g,n}}(\tau_{\gm_0}^{h_0},\ldots,\tau_{\gm_k}^{h_k})$ 
fits in the exact sequence:
$$1\ra Z_{\cGG_{g,n}}(\tau_{\gm_0}^{h_0},\ldots,\tau_{\gm_k}^{h_k})
\ra\cGG_\sg\ra\Sigma_{\sg},$$
where $\cGG_\sg$ is the stabilizer of the simplex 
$\sg$ for the action of $\cGG_{g,n}$ on $L(\hP_{g,n})$
and  $\Sigma_\sg$ is the group of permutations on the set $\sg=\{\gm_0,\ldots,\gm_k\}$.
\end{enumerate}
\end{corollary}

\begin{proof}The first point follows from Corollary~\ref{centralizer twist}, 
Theorem~\ref{Dehn stabilizer} and induction on $k$. The second point from 
Corollary~\ref{centralizer twist}, Theorem~\ref{stab pro-curves} and induction on $k$.

\end{proof} 

For $\sg=\{\gm_0,\ldots,\gm_k\}$ a $k$-simplex of $L(\hP_{g,n})$, let $\Tau_\sg$ be the closed 
abelian subgroup of $\cGG_{g,n}$ spanned by the profinite Dehn twists 
$\tau_{\gm_0},\ldots,\tau_{\gm_k}$. From Corollary~\ref{centralizer2}, it follows:

\begin{corollary}\label{multi-twists}For $2g-2+n>0$, let $\sg=\{\gm_0,\ldots,\gm_s\}$ and
$\sg'=\{\delta_0,\ldots,\delta_t\}$ be two simplices of $L(\hP_{g,n})$. With the above notations,
if the intersection $\Tau_\sg\cap\Tau_{\sg'}$ is open in $\Tau_\sg$, then it holds $\sg\subseteq\sg'$ 
and hence $\Tau_\sg\cap\Tau_{\sg'}=\Tau_\sg$.
\end{corollary}
\begin{proof}By ii.) Corollary~\ref{centralizer2}, since $\Tau_\sg\cap\Tau_{\sg'}$ is open in 
$\Tau_\sg$, the centralizer of the group $\Tau_{\sg'}$ is contained in that of $\Tau_{\sg}$. But then, 
from i) Corollary~\ref{centralizer2}, it follows that there is an inclusion of simplices $\sg\subseteq\sg'$.

\end{proof}

Let us now compute the normalizer of the closed subgroup of $\cGG_{g,n}$ spanned by a 
set of profinite Dehn twists: 

\begin{theorem}\label{normalizer}For $2g-2+n>0$, let be given $\sg=\{\gm_0,\ldots,\gm_k\}$ 
a simplex of $L(\hP_{g,n})$ and $h_0,\ldots,h_k\in\ZZ\ssm\{0\}$. 
Then, the normalizer in $\cGG_{g,n}$ of the closed subgroup generated by the set of powers of 
Dehn twists $\{\tau_{\gm_0}^{h_0},\ldots,\tau_{\gm_k}^{h_k}\}$ is the stabilizer $\cGG_\sg$
of the simplex $\sg$ for the action of $\cGG_{g,n}$ on the complex of profinite curves 
$L(\hP_{g,n})$.
\end{theorem}

\begin{proof}An element of $\cGG_\sg$ normalizes the closed
subgroup generated by $\tau_{\gm_0}^{h_0},\ldots,\tau_{\gm_k}^{h_k}$ of 
$\cGG_{g,n}$. Let then $f\in\cGG_{g,n}$ be an element which also normalize this
subgroup. For all $0\leq i\leq k$, it then holds
$f\tau_{\gm_i}^{h_i}f^{-1}=\tau_{f(\gm_i)}^{h_i}=
\tau_{\gm_0}^{m_0}\cdot\ldots\tau_{\gm_k}^{m_k},$
for some multindex $(m_0,\ldots,m_k)\in(\ZZ)^{k+1}$.
By the following lemma, it then holds that $f\in\cGG_\sg$.

\begin{lemma}\label{multitwist}For $2g-2+n>0$, let be given $\gm\in\hL$, $h\in\ZZ\ssm{0}$,
$\{\gm_0,\ldots,\gm_k\}\in L(\hP_{g,n})$ a simplex and $h_0,\ldots,h_k\in\ZZ$.
Then, there is an identity $\tau_\gm^h=\tau_{\gm_0}^{h_0}\cdot\ldots\tau_{\gm_k}^{h_k}$,
if and only if, for some $0\leq i\leq k$, it holds $\gm_i=\gm$, $h_i=h$ and $h_j=0$, for all $j\neq i$.
\end{lemma}
\begin{proof}It is easy to prove that, since $\{\tau_{\gm_0}^{h_0},\ldots,\tau_{\gm_k}^{h_k}\}$
is a set of commuting elements, if the given identity holds, each of its elements commutes with
$\tau_\gm^h$. Then, the conclusion of the lemma follows from $i.)$ Corollary~\ref{centralizer2}.

\end{proof}\end{proof}

The results of this section also imply that the simplicial profinite complex $L(\hP_{g,n})$ 
parametrizes all the closed abelian subgroups of $\cGG_{g,n}$ spanned by powers of 
profinite Dehn twists. Let us state this more precisely. 

The set of closed subgroups of a profinite group is the inverse limit of the inverse system formed
by the finite sets of subgroups of its finite discrete quotients. Therefore, it is a profinite set.
Let then $\check{\cS}$ be the profinite set of all closed subgroups of $\cGG_{g,n}$. The 
profinite group $\cGG_{g,n}$ acts continuously by conjugation on $\check{\cS}$.
{\it A weight function on $\hL$} is a $\cGG_{g,n}$-equivariant function 
$w\co\hL\ra\N^+$, where we let $\cGG_{g,n}$ act trivially on $\N^+$. 

Let then $\iota(w)_k\co C(S_{g,n})_k\hookra\check{\cS}$ be the natural embedding 
defined sending a $k$-simplex $\{\gm_0,\ldots,\gm_k\}$ to the closed subgroup 
$\ZZ\cdot\tau_{\gm_0}^{w(\gm_0)}\oplus\ldots\oplus\ZZ\cdot\tau_{\gm_k}^{w(\gm_k)}$ of 
$\cGG_{g,n}$ and let $\ol{C}^w(S_{g,n})_k$ be the closure, in the profinite topology, of the image 
of $\iota(w)_k$ in $\check{\cS}$. 

\begin{proposition}\label{group-theoretic}Let $2g-2+n>0$. For every weight function
$w\co\hL\ra\N^+$ and $0\leq k\leq 3g-3+n$, there is a natural 
$\cGG_{g,n}$-equivariant continuous isomorphism:
$$\Delta^w_k\co L(\hP_{g,n})_k\sr{\sim}{\ra}\ol{C}^w(S_{g,n})_k.$$
\end{proposition}

\begin{proof}The map $\Delta^w_k$ exists by the universal property of the $\cGG_{g,n}$-completion.
If, for two simplices $\sg,\sg'\in L(\hP_{g,n})_k$, it holds $\Delta^w_k(\sg)=\Delta^w_k(\sg')$, 
the intersection $\Tau_\sg\cap\Tau_{\sg'}$ is open in both groups $\Tau_\sg$ and $\Tau_{\sg'}$.
From Corollary~\ref{multi-twists}, it then follows that $\sg=\sg'$.

\end{proof}

By Proposition~\ref{group-theoretic}, for every weight function $w\co\hL\ra\N^+$, it is 
defined a simplicial profinite complex $\ol{C}^w(S_{g,n})$,  endowed with a natural, continuous
$\cGG_{g,n}$-action, naturally isomorphic to $L(\hP_{g,n})$, which we call {\it the 
weighted group-theoretic procongruence curve complex}.
So, the above results sum up in an intrinsic virtual description of 
the complex of profinite curves $L(\hP_{g,n})$.

\begin{theorem}\label{intrinsic}Let $2g-2+n>0$. For any given open normal subgroup 
$\cGG^\ld$ of $\cGG_{g,n}$, the simplicial profinite complex $L(\hP_{g,n})$ is naturally
identified with the simplicial complex whose set of $k$-simplices, for $k\geq 0$, consists 
of primitive free $\ZZ$-modules of rank $k+1$ contained in $\cGG^\ld$, generated by powers 
of profinite Dehn twists.  
\end{theorem}

\begin{remark}\label{profinite curves}Observe that $L(\hP_{g,n})$ is a flag complex, i.e. sets of
vertices which are pairwise joinable are joinable. The barycentric subdivision of  
$L(\hP_{g,n})$ is then more simply described as the flag complex associated to the poset
of closed abelian subgroups of $\cGG_{g,n}$ generated by profinite Dehn twists.
\end{remark}

\section{Faithfulness of Galois representations}\label{Galois}
In this section, for $2g-2+n>0$, we denote by $\cM_{g,n}$ the stack of smooth algebraic curves 
defined over some number field and by $\ccM_{g,n}$ its D--M compactification (in contrast with 
Sections~\ref{levels}--\ref{centralizers}, where we considered only complex algebraic curves). 
Both stacks are defined over $\mathrm{Spec}(\Q)$.

Let $C$ be a smooth $n$-punctured, genus $g$ curve, defined over a number field $\K$. To give 
such a curve, it is equivalent to give a point $\xi\co\mathrm{Spec}(\K)\ra\cM_{g,n}$. If 
$\cC\ra\cM_{g,n}$ denotes the universal $n$-punctured, genus $g$ curve, the curve $C$ is 
isomorphic to the fiber $\cC_\xi$. Let $\ol{\xi}$ be the geometric point of $\cM_{g,n}$ associated to 
$\xi$ and a given embedding $\K\subset\ol{\Q}$, let $\tilde{\xi}$ be a closed point of 
$C\times_\K\ol{\Q}$ (which we identify with $\cC_{\ol{\xi}}$) and let $G_\Q$ be the absolute 
Galois group. There is a commutative exact diagram:

$$\begin{array}{ccccccrr}
&1&&1&&&&\\
&\da&&\da&&&&\\
1\ra&\pi_1(C\times_\K\ol{\Q},\tilde{\xi})&\sr{\sim}{\ra}&\pi_1(C\times_\K\ol{\Q},\tilde{\xi})&\ra&1&\\
&\da&&\da&&\da&&\\
1\ra&\pi_1(\cC\times\ol{\Q},\tilde{\xi})&\ra&\pi_1(\cC,\tilde{\xi})&\ra&G_\Q&\ra 1
&\hspace{0.7cm}(7.1)\\
&\da&&\da&&\,\,\da\!\wr&&\\
1\ra&\pi_1(\cM_{g,n}\times\ol{\Q},\ol{\xi})&\ra&\pi_1(\cM_{g,n},\ol{\xi})&\sr{p}{\ra}&G_\Q&\ra 1&\\
&\da&&\da&&\da&&\\
&1&&1&&1&.&
\end{array}
$$

Let $G_\K$ be the absolute Galois group of $\K$. Then, the group $G_\K$ identifies with an open
subgroup of $G_\Q$ and the algebraic fundamental group $\pi_1(\cM_{g,n}\times\K,\ol{\xi})$ with
$p^{-1}(G_\K)$. 

Moreover, the point $\xi\in\cM_{g,n}$ induces a homomorphism on algebraic fundamental groups
$s_\xi\co G_\K\ra\pi_1(\cM_{g,n}\times\K,\ol{\xi})$ which is a canonical section of the natural 
epimorphism $p\co\pi_1(\cM_{g,n}\times\K,\ol{\xi})\tura G_\K$.

There are various representations associated to the above diagram.
To the middle column, is associated {\it the arithmetic universal monodromy representation}:
$$\mu_{g,n}\co\pi_1(\cM_{g,n},\ol{\xi})\ra\out(\pi_1(C\times_\K\ol{\Q},\tilde{\xi})),\hspace*{2cm}(7.2)$$
which encloses all the others. Composing this representation with the section $s_\xi$, we get indeed 
the outer monodromy representation associated to the curve $C/\K$:
$$\rho_C\co G_\K\ra\out(\pi_1(C\times_\K\ol{\Q},\tilde{\xi})).\hspace*{4cm}(7.3)$$
If, instead, we compose the arithmetic universal monodromy representation with the 
natural monomorphism $\pi_1(\cM_{g,n}\times\ol{\Q},\ol{\xi})\hookra\pi_1(\cM_{g,n},\ol{\xi})$, 
we get {\it the geometric universal monodromy representation}:
$$\ol{\mu}_{g,n}\co\pi_1(\cM_{g,n}\times\ol{\Q},\ol{\xi})\ra
\out(\pi_1(C\times_\K\ol{\Q},\tilde{\xi})).\hspace*{2cm}(7.4)$$

Let us fix an embedding $\ol{\Q}\subset\C$ and let us assume that, with the notations of 
Section~\ref{levels}, it holds $\ol{\xi}=a$ and $\tilde{\xi}=\tilde{a}$. Then, there are natural
isomorphisms $\hGG_{g,n}\cong\pi_1(\cM_{g,n}\times\ol{\Q},\ol{\xi})$,
$\hGG_{g,n+1}\cong\pi_1(\cM_{g,n+1}\times\ol{\Q},\tilde{\xi})$ and 
$\hP_{g,n}\cong\pi_1(C\times_\K\ol{\Q},\tilde{\xi})$, and the geometric universal monodromy 
representation identifies with the profinite universal monodromy representation 
$\hat{\rho}_{g,n}$ introduced at the beginning of Section~\ref{completions}.

In this section, we are going to prove that the representation $(7.3)$ is faithful for all hyperbolic curves 
$C$, thus reproving a result obtained by Matsumoto \cite{Matsu} in the affine case and 
extended, more recently, by Hoshi and Mochizuki \cite{H-M} to the projective case.

Moreover, we will show that the representation $(7.2)$ is faithful if and only the representation $(7.4)$
is faithful, i.e. if the subgroup congruence property holds for $\GG_{g,n}$ (this result appeared as well
in \cite{H-M}). In particular, by Theorem~2.10 in \cite{Hyp2}, it will follow that $\mu_{g,n}$ is faithful for 
$2g-2+n>0$ and $g\leq 2$. Let us give the following definition:

\begin{definition}\label{star condition}Let $\GG'_{g,n}$, for $2g-2+n>0$, be a profinite completion of 
the Teichm\"uller group. We denote by $\aut^\ast(\GG'_{g,n})$ the group of automorphisms of 
$\GG'_{g,n}$ which preserve the set of closed cyclic subgroups of $\GG'_{g,n}$ generated by 
profinite Dehn twists. From the definition of profinite Dehn twists, it follows that 
$\mathrm{Inn}(\GG'_{g,n})\unlhd\aut^\ast(\GG'_{g,n})$. Let then $\out^\ast(\GG'_{g,n})$ 
be the quotient of these two groups.
\end{definition}

Let us identify $\hGG_{g,n}$ with the normal subgroup $\pi_1(\cM_{g,n}\times\ol{\Q},\ol{\xi})$
of $\pi_1(\cM_{g,n},\ol{\xi})$. Restriction of inner automorphisms then defines the representation:
$$\hat{\cG}_{g,n}\co\pi_1(\cM_{g,n},\ol{\xi})\ra\aut^\ast(\hGG_{g,n}).$$

For $2g-2+n>0$, the procongruence Teichm\"uller group $\cGG_{g,n}$ is defined as the image 
of the geometric universal monodromy representation $\ol{\mu}_{g,n}$. Thus, it is a normal 
subgroup of $\Im\,\mu_{g,n}$ and $\pi_1(\cM_{g,n},\ol{\xi})$ acts on it composing the 
representation $\mu_{g,n}$ with the action by conjugation. 
So, we get the natural representation, compatible with $\hat{\cG}_{g,n}$:
$$\check{\cG}_{g,n}\co\pi_1(\cM_{g,n},\ol{\xi})\ra\aut^\ast(\cGG_{g,n}).$$

We will see that, thanks to the results of sections \ref{profinite twists} and \ref{centralizers}, 
the representation $\check{\cG}_{g,n}$, while 
containing all the arithmetic information of the representation $\hat{\cG}_{g,n}$, is
much more treatable. For the moment, let us observe that, by lifting elements of $G_\Q$ to 
$\pi_1(\cM_{g,n},\ol{\xi})$, we get also natural compatible representations, for $2g-2+n>0$:
$$\hat{b}_{g,n}\co G_\Q\ra\out^\ast(\hGG_{g,n})\hspace{1cm}\mbox{and}\hspace{1cm}
\check{b}_{g,n}\co G_\Q\ra\out^\ast(\cGG_{g,n}).$$

By the characterization given in Theorem~\ref{intrinsic} of the complex of profinite curves 
$L(\hP_{g,n})$, there is a natural continuous representation of $\aut^\ast(\cGG_{g,n})$ into 
$\aut(L(\hP_{g,n}))$, the group of continuous automorphisms of the simplicial profinite complex 
$L(\hP_{g,n})$. This representation is faithful:

\begin{theorem}\label{faithful}For $2g-2+n>0$, there is a natural faithful representation:
$$\aut^\ast(\cGG_{g,n})\hookra\aut(L(\hP_{g,n})).$$
\end{theorem}

\begin{proof}Let $\phi\co\cGG_{g,n}\ra\cGG_{g,n}$ be an automorphism such that it holds 
$\phi(\ol{\langle\tau_\gm\rangle})=\ol{\langle\tau_\gm\rangle}$, 
for every profinite Dehn twist $\tau_\gm\in\cGG_{g,n}$, i.e., for every
$\gm\in L(\hP_{g,n})_0$, it holds $\phi(\tau_\gm)=\tau_\gm^{\alpha_\gm}$, for some 
$\alpha_\gm\in\ZZ^\ast$. Let us then show that $\phi$ is the identity.

Let $\gm_0$, for $g\geq 1$ (resp. $g=0$), be a non-separating s.c.c. (resp. a s.c.c. bounding
a disc with two punctures) on $S_{g,n}$. The associated Dehn twists generate topologically 
$\cGG_{g,n}$. So, if we prove that $\phi(\tau_{\gm_0})=\tau_{\gm_0}$, it will follow that $\phi$
is the identity.

Let then $\gm_1$, for $g\geq 1$ (resp. $g=0$), be a s.c.c. of the same type of $\gm_0$ and 
which intersects $\gm_0$, geometrically, only once (resp. only twice). The corresponding Dehn 
twists then satisfy the braid relation 
$\tau_{\gm_0}\tau_{\gm_1}\tau_{\gm_0}^{-1}=\tau_{\gm_1}^{-1}\tau_{\gm_0}\tau_{\gm_1}$.
Applying the automorphism $\phi$, we get the identity:
$$\tau_{\gm_0}^{\alpha_{\gm_0}}\tau_{\gm_1}^{\alpha_{\gm_1}}\tau_{\gm_0}^{-\alpha_{\gm_0}}=
\tau_{\gm_1}^{-\alpha_{\gm_1}}\tau_{\gm_0}^{\alpha_{\gm_0}}\tau_{\gm_1}^{\alpha_{\gm_1}},$$
or, equivalently:
$$\tau_{\tau_{\gm_0}^{\alpha_{\gm_0}}(\gm_1)}^{\alpha_{\gm_1}}=
\tau_{\tau_{\gm_1}^{-\alpha_{\gm_1}}(\gm_0)}^{\alpha_{\gm_0}}.$$
By Remark~\ref{separability}, it holds $\alpha_{\gm_0}=\alpha_{\gm_1}$ 
(let us then call $u$ this element) 
and $\tau_{\gm_0}^{u}(\gm_1)=\tau_{\gm_1}^{-u}(\gm_0)$ in $\hL$. 
For suitable lifts $\td{\gm}_0$ and $\td{\gm}_1\in\cL^\ast_{g,n}$ of  
$\gm_0$ and $\gm_1$, respectively, the loop $\td{\gm}_1\cdot\td{\gm}_0^{u}$ 
is then conjugated in $\hP_{g,n}$ either to the loop $\td{\gm}_0\cdot\td{\gm}_1^{u}$ or to the loop
$\td{\gm}_1^{-u}\cdot\td{\gm}_0^{-1}$. But now, since the classes of $\td{\gm}_0$ and $\td{\gm}_1$ 
in the homology group $H_1(S_{g,n},\ZZ)$ are linearly independent, this is possible only for $u=1$.

\end{proof}

For an open normal subgroup $\cGG^\ld$ of $\cGG_{g,n}$, let us define $\aut^\ast(\cGG^\ld)$
to be the group of automorphisms of $\cGG^\ld$ which preserve the set of primitive closed cyclic 
subgroups generated by powers of profinite Dehn twists. 

\begin{corollary}\label{faithful2}For $2g-2+n>0$, let $\cGG^\ld$ be a characteristic open 
subgroup of $\cGG_{g,n}$. Then, the natural representation 
$\aut^\ast(\cGG_{g,n})\ra\aut^\ast(\cGG^\ld)$, induced by the restriction of automorphisms, is faithful. 
\end{corollary}
\begin{proof}By the characterization given in Theorem~\ref{intrinsic} of the complex of profinite curves 
$L(\hP_{g,n})$, this representation factorizes the faithful representation of Theorem~\ref{faithful}:
$$\aut^\ast(\cGG_{g,n})\ra\aut^\ast(\cGG^\ld)\ra\aut(L(\hP_{g,n})).$$
Hence, it is faithful as well.

\end{proof}

The next theorem is a consequence of the description of centralizers of Dehn twists given in 
Section~\ref{centralizers}. It is also the key to implement the Grothendieck-Teichm\"uller 
Lego in higher genus, on which all the following faithfulness results of this section are based.

\begin{theorem}\label{outer}Let $(g,n)$ and $(g',n')$ be such that $2g-2+n>0$, $2g'-2+n'>0$, 
$g\geq g'$ and $3g-3+n> 3g'-3+n'$. Then, there is a natural $G_\Q$-equivariant homomorphism:
$$\out^\ast(\cGG_{g,n})\ra\out^\ast(\cGG_{g',n'}).$$
\end{theorem}

\begin{proof}By Theorem~\ref{faithful}, the natural representation 
$\aut^\ast(\cGG_{g,n})\ra\aut(L(\hP_{g,n}))$ is faithful. A natural guess then is that  
the action of $\aut^\ast(\cGG_{g,n})$ on the complex of profinite curves also preserves
the topological type of the simplices. We actually need and we will prove here a 
slightly weaker assertion:

\begin{lemma}\label{prestype}For $2g-2+n>0$, the action of $\aut^\ast(\cGG_{g,n})$ on $\hL$ 
preserves the topological type of non-separating profinite s.c.c.'s and of profinite s.c.c.'s 
bounding a disc with two punctures.
\end{lemma}

\begin{proof}Let us denote by $\hL_0$ the $\cGG_{g,n}$-orbit in $\hL$ consisting of non-separating 
profinite s.c.c.'s and, for $n\geq 2$, by $\hL_1$ the $\cGG_{g,n}$-orbit in $\hL$ consisting of profinite 
s.c.c.'s bounding a disc with two punctures. Let us show first that the action of $\aut^\ast(\cGG_{g,n})$ 
on $\hL$ preserves the union $\hL_0\cup\hL_1$.

A profinite s.c.c. $\gm\in\hL$ is in the $\cGG_{g,n}$-orbit of a discrete one and,
for all $f\in\aut^\ast(\cGG_{g,n})$, there is an isomorphism:
$$Z_{\cGG_{g,n}}(\tau_\gm)\cong Z_{\cGG_{g,n}}(\tau_{f(\gm)}).$$

By Corollary~\ref{centralizer2}, any pair of commuting profinite Dehn twists in 
$Z_{\cGG_{g,n}}(\tau_\gm)$, for $\gm\in\cL$ discrete, is conjugated, inside this group, 
to a pair of commuting discrete Dehn twists.

Therefore, the centralizer of a profinite Dehn twist 
$\tau_\gm$, for $\gm\in\hL_0\cup\hL_1$, is characterized by
the property that, for any pair $\tau_\alpha,\tau_\beta\in Z_{\cGG_{g,n}}(\tau_\gm)$ of commuting 
profinite Dehn twists distinct from $\tau_\gm$, there is a profinite Dehn twist $\tau_{\delta}$ in the 
group $Z_{\cGG_{g,n}}(\tau_\gm)$ which commutes neither with
$\tau_\alpha$ nor with $\tau_\beta$. 

The above argument already proves the lemma for $n\leq 1$. So, let us assume $n\geq 2$.
We need only to prove that, for any $\gm\in\hL_0$ and all $f\in\aut^\ast(\cGG_{g,n})$, it holds 
$f(\gm)\notin\hL_1$. 

Let us observe that the set of non-separating profinite Dehn twists forms a single
conjugacy class in $\cGG_{g,n}$, containing at least two distinct commuting elements for $n\geq 2$.
On the other hand, to each pair of punctures on $S_{g,n}$, corresponds a distinct conjugacy class in 
$\cGG_{g,n}$ of profinite Dehn twists along profinite s.c.c.'s bounding a disc with two punctures and 
none of these conjugacy classes contains a pair of distinct commuting elements.

Every automorphism $f$ of $\cGG_{g,n}$ sends conjugacy classes to conjugacy classes and pairs
of commuting elements to pairs of commuting elements.  Therefore, $\gm\in\hL_0$ implies
$f(\gm)\notin\hL_1$. 

\end{proof}  

In order to prove Theorem~\ref{outer}, let us split it in the three assertions:
\begin{enumerate} 
\item for $g\geq 1$, there is a natural homomorphism 
$\out^\ast(\cGG_{g,n})\ra\out^\ast(\cGG_{g-1,n+2})$;
\item for $n\geq 2$, there is a natural homomorphism 
$\out^\ast(\cGG_{g,n})\ra\out^\ast(\cGG_{g,n-1})$; 
\item for $g\geq 2$, there is a natural homomorphism 
$\out^\ast(\cGG_{g,1})\ra\out^\ast(\cGG_g)$. 
\end{enumerate}
Galois equivariance then follows from the naturality of the above homomorphisms.

In order to prove i.), let us choose an orientation for every element of $\hat{\cL}$ and let us fix
a non-separating s.c.c. $\gm$ on $S_{g,n}$. By Lemma~\ref{prestype}, for any given 
$\tilde{f}\in\aut^\ast(\cGG_{g,n})$, there is an $h\in\cGG_{g,n}$ such that 
$\tilde{f}(\tau_{\gm})=\inn(h)(\tau_{\gm})$. Moreover, since, for a non-separating s.c.c. 
$\gm$ on $S_{g,n}$, there is a homeomorphism preserving $\gm$ but switching its orientations, 
we can choose $h$ such that $h(\vec{\gm})$ has the fixed orientation. 

Therefore, given an $f\in\out^\ast(\cGG_{g,n})$, there is a lift $\tilde{f}\in\aut^\ast(\cGG_{g,n})$ such 
that $\tilde{f}(\tau_\gm)=\tau_\gm$. By Theorem~\ref{Dehn stabilizer},  the restriction $\tilde{f}_\gm$ 
of the automorphism $\tilde{f}$ to the subgroup $\cGG_\gm=Z_{\cGG_{g,n}}(\tau_\gm)$ preserves as 
well the set of closed cyclic groups generated by the profinite Dehn twists of $\cGG_\gm$, i.e. the 
closure of the set of closed cyclic groups generated by the Dehn twists of $\GG_\gm$. 

From the description of $\cGG_\gm$ given in Corollary~\ref{centralizer twist},  
it follows that $\cGG_{\vec{\gm}}$ is the closed normal subgroup of $\cGG_\gm$
topologically generated by the Dehn twists contained in there. Hence, the 
automorphism $\tilde{f}_\gm$ can be further restricted to the subgroup $\cGG_{\vec{\gm}}$. 

By Theorem~\ref{stab pro-curves} and 
Corollary~\ref{center-free}, there is a natural isomorphism 
$\cGG_{\vec{\gm}}/Z(\cGG_{\vec{\gm}})\cong\cGG_{g-1,n+2}$.
It then follows that $\tilde{f}_\gm$ induces an automorphism $\ol{f}_\gm$ of $\cGG_{g-1,n+2}$.

Let $\tilde{f}'\in\aut^\ast(\cGG_{g,n})$ be another lift of $f$ with the same properties and let 
$\ol{f}'_\gm$ be the automorphism it induces on $\cGG_{g-1,n+2}$. We then have to prove 
that $\ol{f}_\gm$ and $\ol{f}'_\gm$ differ by an inner automorphism of $\cGG_{g-1,n+2}$. 

By definition, the product $\tilde{f}^{-1}\tilde{f}'$ is an inner automorphism $\inn(x)$ of 
$\cGG_{g,n}$, such that $\inn(x)(\tau_\gm)=\tau_\gm$. By Corollary~\ref{centralizer twist}, 
$x$ stabilizes $\gm$ but then also preserves its chosen orientation. Therefore, it holds
$x\in\cGG_{\vec{\gm}}$ and the claim follows.

Item ii.) is proved with a similar argument to that used for i.), where we let $\gm$ be instead a s.c.c. 
on $S_{g,n}$ bounding a two-punctured disc.

Let us now prove item iii.). 
We just need to show that any element of $\aut^\ast(\cGG_{g,1})$ preserves the normal subgroup 
$\hP_g$ of $\cGG_{g,1}$. As a subgroup of $\cGG_{g,1}$, the group $\hP_g$ 
is topologically generated by products of Dehn twists of the form $\tau_{\gm_0}\tau_{\gm_1}^{-1}$, 
where $\{\gm_0,\gm_1\}$ is a cut pair on $S_{g,1}$ consisting of non-separating s.c.c.'s bounding 
a cylinder containing the puncture. Therefore, it is enough to prove that elements of 
$\aut^\ast(\cGG_{g,1})$ preserve the topological type of the $1$-simplices of $L(\hP_{g,1})$ 
corresponding to cut pairs on $S_{g,1}$ of the above type.

So, given such a cut pair $\{\gm_0,\gm_1\}$ on $S_{g,1}$ and any $f\in\aut^\ast(\cGG_{g,1})$, let us 
show that $\{f(\gm_0),f(\gm_1)\}$ has the same topological type. In order to prove this assertion, we 
are allowed to modify the given $f$ by an inner automorphism of $\cGG_{g,1}$. 

Since $\gm_0$ and $\gm_1$ are non-separating, by Lemma~\ref{prestype} and the fact that
non-separating profinite s.c.c.'s form a single conjugacy class in $\cGG_{g,1}$, we
may then assume that $f(\gm_0)=\gm_0$, that $f$ restricts to an automorphism of 
$\cGG_{\vec{\gm}_0}$ and that $f(\gm_1)$ is non-separating.

The s.c.c. $\gm_1$ determines on $S_{g,1}\ssm\gm_0$ a separating s.c.c. bounding a disc 
with two punctures, one of which is the only puncture on $S_{g,1}$ 
and the other is one of the two determined by the cutting along $\gm_0$. 
Then, in virtue of Lemma~\ref{prestype} applied to the group $\aut^\ast(\cGG_{\vec{\gm}_0})$, the 
profinite s.c.c. $f(\gm_1)$ bounds also a disc containing two punctures on the cut surface 
$S_{g,1}\ssm\gm_0$. Now, one of these punctures must be the one coming from $S_{g,1}$, 
otherwise $f(\gm_1)$ would bound a genus $1$ subsurface on $S_{g,1}$. 
Therefore, $\{f(\gm_0),f(\gm_1)\}$ is a cut pair of the same topological type of $\{\gm_0,\gm_1\}$.

\end{proof}

\begin{corollary}\label{galois-faithful}For $2g-2+n>0$ and $3g-3+n>0$, the natural
representation $\check{b}_{g,n}\co G_\Q\ra\out^\ast(\cGG_{g,n})$ is faithful. In particular, the 
representation $\hat{b}_{g,n}\co G_\Q\ra\out^\ast(\hGG_{g,n})$ is faithful as well.
\end{corollary}

\begin{proof}For $(g,n)=(0,4),(1,1)$ the above statement is well known and a direct consequence 
of Bely\v{\i}'s Theorem (cf. \cite{Belyi}). The general case then follows, since, by 
Theorem~\ref{outer}, for $3g-3+n>1$, there is a natural $G_\Q$-equivariant homomorphism 
$\out^\ast(\cGG_{g,n})\ra\out^\ast(\cGG_{0,4})$.

\end{proof}

As an almost immediate consequence, we get the most important faithfulness result of this section:

\begin{theorem}\label{outer-galois}Let $C$ be a hyperbolic curve defined over a number field $\K$.
Then, the associated outer Galois representation 
$\rho_C\co G_\K\ra\out(\pi_1(C\times_\K\ol{\Q},\tilde{\xi}))$ 
is faithful.
\end{theorem}

\begin{proof}Let us assume that $C$ is the curve which appears in the diagram $(7.1)$.
From Corollary~\ref{galois-faithful}, it follows that the homomorphism
$\check{\cG}_{g,n}\circ s_\xi\co G_\K\ra\aut^\ast(\cGG_{g,n})$ is injective. This homomorphism 
can be described as follows. For $h\in G_\K$, the automorphism $\check{\cG}_{g,n}\circ s_\xi(h)$ of 
$\cGG_{g,n}<\out(\pi_1(C\times_\K\ol{\Q}))$ is defined by the assignment, for $f\in\cGG_{g,n}$:
$$f\longmapsto\rho_C(h)\cdot f\cdot\rho_C(h)^{-1}\in\out(\pi_1(C\times_\K\ol{\Q})).$$
Therefore, it follows that $\ker\,\rho_C\leq\ker(\check{\cG}_{g,n}\circ s_\xi)=\{1\}.$

\end{proof}

\begin{remark}\label{tangential}For suitable choices of tangential base point $\xi$, for instance, 
if $\cC_{\ol{\xi}}$ is a nodal curve with an elliptic tail defined over $\Q$, it is not difficult 
to infer directly from Bely\v{\i}'s Theorem that the associated representation 
$\check{\cG}_{g,n}\circ s_\xi\co G_\Q\ra\aut^\ast(\cGG_{g,n})$ is faithful.
However, this does not imply, at least not formally, the faithfulness of the representation 
$\check{\cG}_{g,n}\circ s_\xi$ for any other choice of base point $\xi$.
\end{remark}

A $\Q$-rational point of the moduli stack $\ccM_{g,n}$, given for instance by a suitable graph of 
punctured projective lines, determines a $\Q$-rational tangential base point of $\cM_{g,n}$ and so
a $\pi_1(\cM_{g,n}\times\ol{\Q},\ol{\xi})$-conjugacy class of splittings 
of the short exact sequence:
$$1\ra\pi_1(\cM_{g,n}\times\ol{\Q},\ol{\xi})\ra\pi_1(\cM_{g,n},\ol{\xi})\sr{p}{\ra}G_\Q\ra 1.$$
Let $\sg\co G _\Q\ra\pi_1(\cM_{g,n},\ol{\xi})$ be a section of $p$ corresponding to one of
these splittings. It determines an isomorphism:
$$\pi_1(\cM_{g,n},\ol{\xi})\cong\hGG_{g,n}\rtimes_{\td{b}_{g,n}} G_\Q,$$
where we have identified $\pi_1(\cM_{g,n}\times\ol{\Q},\ol{\xi})$ with $\hGG_{g,n}$ and $G_\Q$
acts on it by means of the representation 
$\td{b}_{g,n}:=\hat{\cG}_{g,n}\circ\sg\co G_\Q\ra\aut^\ast(\hGG_{g,n})$. 

In terms of this 
isomorphism, the representation $\hat{\cG}_{g,n}\co\pi_1(\cM_{g,n},\ol{\xi})\ra\aut^\ast(\hGG_{g,n})$ 
is then described by sending an element $(f,h)\in\hGG_{g,n}\rtimes_\sg G_\Q$ to the automorphism of 
$\hGG_{g,n}$ given by the composition $\inn\,f\cdot\td{b}_{g,n}(h)$.

\begin{proposition}\label{GT}\begin{enumerate}
\item For $2g-2+n>0$, the kernel of the natural representation
$\hat{\cG}_{g,n}\co\pi_1(\cM_{g,n},\ol{\xi})\ra\aut^\ast(\hGG_{g,n})$ can be identified with
the center of $\hGG_{g,n}$. 

In particular, for $g\leq 2$, the representation $\hat{\cG}_{g,n}$ is faithful for $(g,n)\neq (1,1),(2,0)$.
Otherwise its kernel is spanned by the hyperelliptic involution.

\item For $g\geq 3$ and $n\geq 0$, the kernel of the representation
$\check{\cG}_{g,n}\co\pi_1(\cM_{g,n},\ol{\xi})\ra\aut^\ast(\cGG_{g,n})$ can be identified with
the congruence kernel of $\hGG_{g,n}$.
\end{enumerate}
\end{proposition}

\begin{proof}From Corollary~\ref{galois-faithful}, it follows that $\td{b}_{g,n}$ is faithful and that 
$\mathrm{Inn}(\hGG_{g,n})\cap\Im\,\td{b}_{g,n}=\{1\}$. These two facts yield the first 
statement of $i.)$. 

By Proposition~3.2 in \cite{Hyp2}, for $g\leq 2$, it holds $Z(\hGG_{g,n})=Z(\GG_{g,n})$ and 
this yields the second statement of $i.)$. 

As to $ii.)$, it follows from the fact that, by definition, the kernel of the natural homomorphism 
$\hGG_{g,n}\ra\cGG_{g,n}$ is the congruence kernel, by Corollary~\ref{center-free}, for $g\geq 3$, 
the natural representation $\mathrm{inn}\co\cGG_{g,n}\ra\aut^\ast(\cGG_{g,n})$ is faithful and,
by Corollary~\ref{galois-faithful}, the composition $\td{b}_{g,n}':=\check{\cG}_{g,n}\circ\sg$ 
is faithful and it holds $\mathrm{Inn}(\cGG_{g,n})\cap\Im\,\td{b}_{g,n}'=\{1\}$.

\end{proof}

Since the congruence subgroup property holds in genus $\leq 2$ (cf. \cite{Asada}, \cite{Hyp}, 
\cite{Hyp2}), from the same argument used in the proof of Theorem~\ref{outer-galois}, it then follows:

\begin{corollary}\label{universal-faithful}Let $2g-2+n>0$. The arithmetic universal monodromy 
representation $\mu_{g,n}\co\pi_1(\cM_{g,n},\ol{\xi})\ra\out(\pi_1(C\times_\K\ol{\Q},\tilde{\xi}))$ 
is faithful for $g\leq 2$. Otherwise its kernel can be identified with the congruence kernel of 
the profinite Teichm\"uller group $\hGG_{g,n}$.
\end{corollary}

\subsection*{Acknowledgements}
This work was begun during my stay at the University of Costa Rica in San Jos\'e and completed 
during my stay at the Department of Mathematics of the University of los Andes in Bogot\'a.
I thank both institutions for their support. I also thank Shinichi Mochizuki who commented on a previous 
version of this paper and suggested the statement of Lemma~\ref{mochizuki}.

\bigskip

\noindent Address:\, Departamento de Matem\'aticas, Universidad de los Andes, \\
\hspace*{1.7cm}  Carrera $1^a$ $\mathrm{N}^o$ 18A-10, Bogot\'a, Colombia.
\\
E--mail:\,\,\, marco.boggi@gmail.com

\end{document}